\documentclass[11pt,reqno]{amsart}

\usepackage{graphicx}
\usepackage[top=20mm, bottom=20mm, left=30mm, right=30mm]{geometry}
\usepackage{amsmath,amssymb,amsfonts,amscd} 
\usepackage{fancyvrb}
\usepackage{ulem} 
\usepackage{url} 
\usepackage{fancybox} 
\usepackage{empheq} 
\usepackage{mdframed} 
\usepackage{thmtools} 
\usepackage{mboxfill} 

\usepackage{relsize} 

\theoremstyle{plain}
\newtheorem{theorem}{Theorem}[section]
\newtheorem{lemma}[theorem]{Lemma}

\newtheorem{cor}[theorem]{Corollary}
\newtheorem{prop}[theorem]{Proposition}

\theoremstyle{definition}
\newtheorem{definition}[theorem]{Definition}
\newtheorem{example}[theorem]{Example}

\theoremstyle{remark}
\theoremstyle{definition}
\newtheorem{remark}[theorem]{\textit{Remark}}

\usepackage[only,boxdot,boxast,llbracket,rrbracket]{stmaryrd}

\usepackage{ifthen} 
\provideboolean{True} 
\setboolean{True}{true} 
\provideboolean{False} 
\setboolean{False}{false} 

\usepackage{caption} 
\newcommand{\setcaption}[3]{ 
     \bigskip\hrule\bigskip 
     \addtocounter{table}{-1} 
     \captionlistentry*[#1]{(#2)} 
     \caption[#1]{#2 \\ \smaller\it{#3}} 
     \smallskip 
     \addtocounter{table}{1} 
} 

\usepackage[usenames,dvipsnames]{xcolor} 
\usepackage{longtable} 
\usepackage{tabu} 
\usepackage{colortbl} 

\newcolumntype{T}[2]{>{\relsize{#1}}#2<{}}
\def\TblW{1.0\textwidth} 

\newcounter{tbleqn}[table] 
\newenvironment{teq}{% 
     \refstepcounter{tbleqn} 
}{% 
}

\newcommand{\TagAndLabelTableEqn}[2]{ 
     \begin{teq}%{#1}{#2} 
        \label{#1} 
        \end{teq} 
        \ifthenelse{\boolean{#2}}{\small{\eqref{#1}}}{ } 
} 
\newcommand{\LabelTableEqn}[1]{\TagAndLabelTableEqn{#1}{True}}

\colorlet{DefEmphTextColor}{Blue!85!Gray} 
\colorlet{OTblBC}{Plum!85!Black} 
\colorlet{TblHdrBg}{Plum!50!White} 
\colorlet{ITblBC}{Plum!100!White} 

\renewcommand{\emph}[1]{\textit{#1}}

\newcommand{\keywordemph}[1]{{\color{TealBlue!100!White}{{\emph{#1}}}}} 
\renewcommand{\keywordemph}[1]{{\color{Magenta!100!White}{{\emph{#1}}}}} 

\renewcommand{\keywordemph}[1]{\emph{#1}}
\newcommand{\emphonce}[1]{\emph{#1}} 
\newcommand{\Section}[0]{\S} 
\newcommand{\sref}[1]{Section\ \ref{#1}} 

\newcommand{\cf}[0]{\textit{cf}.\ } 
\newcommand{\ie}[0]{\textit{i.e.},\ }

\usepackage{cancel} 
\newcommand{\nequiv}[0]{\ensuremath{\neq}}

\newcommand{\gkpSII}[2]{\ensuremath{\genfrac{\lbrace}{\rbrace}{0pt}{}{#1}{#2}}} 
\newcommand{\Iverson}[1]{\ensuremath{\left[#1\right]_{\delta}}}

\newcommand{\QPochhammer}[3]{\ensuremath{\left(#1; #2\right)_{#3}}} 
\newcommand{\QP}[1]{\ensuremath{\left(#1\right)_{\infty}}} 
\newcommand{\EGFSeriesTransform}[2]{\ensuremath{\EGF_{#1}\left[\left[#2\right]\right]}} 
\newcommand{\OGFSeriesTransform}[2]{\ensuremath{\OGF_{#1}\left[\left[#2\right]\right]}} 

\renewcommand{\Re}{\RePart} 
\renewcommand{\Im}{\ImPart} 

\DeclareMathOperator{\sq}{sq} 
 
\DeclareMathOperator{\scFn}{sc} 
 
\DeclareMathOperator{\ImPart}{Im} 
\DeclareMathOperator{\RePart}{Re} 
\DeclareMathOperator{\Log}{Log} 
\DeclareMathOperator{\EGF}{EGF} 
\DeclareMathOperator{\OGF}{OGF} 
\DeclareMathOperator{\Li}{Li} 
\DeclareMathOperator{\NumFn}{Num} 

\newcommand{\StartGroupingSubEquations}{\begin{subequations}} 
\newcommand{\EndGroupingSubEquations}{\end{subequations}} 

\newcommand{\citep}{\cite}

\title[Square Series Generating Function Transformations]{
       Square Series Generating Function Transformations}  

\author[Maxie D. Schmidt]{Maxie D. Schmidt \\ \\ 
        School of Mathematics \\ 
        Georgia Institute of Technology \\ 
        117 Skiles Building \\ 
        686 Cherry Street NW \\ 
        Atlanta, GA 30332 \\ \\ 
        \texttt{maxieds@gmail.com}}         
\address{School of Mathematics, Georgia Institute of Technology, Atlanta, GA 30332}
\email{maxieds@gmail.com} 
\thanks{} 

\date{2017.03.24-v1} 

\subjclass[2010]{05A15; 44A99; 33E99; 11B73.} 

\keywords{
Square series, generating function, series transformation, 
integral representation, gamma function, double factorial, 
Stirling number of the second kind, 
Jacobi theta function, Ramanujan theta function. 
} 

\allowdisplaybreaks

\begin{document}

\begin{abstract} 
We construct new integral representations for transformations of the 
ordinary generating function for a sequence, $\langle f_n \rangle$, 
into the form of a generating function that enumerates the 
corresponding ``square series'' generating function for the 
sequence, $\langle q^{n^2} f_n \rangle$, 
at an initially fixed non-zero $q \in \mathbb{C}$. 
The new results proved in the article 
are given by integral--based 
transformations of ordinary generating function series 
expanded in terms of the Stirling numbers of the second kind. 
We then employ known integral representations for the gamma and 
double factorial functions 
in the construction of these square series transformation integrals. 
The results proved in the article lead to new applications and 
integral representations for special function series, 
sequence generating functions, and other related applications. 
A summary \emph{Mathematica} notebook providing derivations of key results and 
applications to specific series is provided online as a supplemental 
reference to readers. 
\end{abstract}

\newpage\maketitle

\section{Notation and Conventions} 
\label{subSection_Intro_Notation_Conventions} 
\label{Section_PostIntro_Notation_Conventions} 

Most of the notational conventions within the article are 
consistent with those employed in the references \citep{GKP,NISTHB}. 
Additional notation for special parameterized classes of the 
\emph{square series} expansions studied in the article is defined in 
Table \ref{table_Notation_SeriesDefsOfSpecialSqSeries-stmts_v1} on page 
\pageref{table_Notation_SeriesDefsOfSpecialSqSeries-stmts_v1}. 
We utilize this notation for these generalized classes of square series 
functions throughout the article. 
The following list provides a description of the other 
primary notations and related conventions employed throughout the 
article specific to the handling of sequences and the 
coefficients of formal power series: 
\begin{itemize} 
     \newcommand{\itemmarker}{\tiny{$\blacktriangleright$}} 
     \newcommand{\itemlabel}[1]{\textbf{#1}:} 
     \newcommand{\itemII}[1]{\item[\itemmarker] \itemlabel{#1} \\ } 
     \renewcommand{\itemII}[1]{\item[\itemmarker] \itemlabel{#1}} 

\itemII{Sequences and Generating Functions} 
The notation $\langle f_n \rangle \equiv \{f_0, f_1, f_2, \ldots\}$ is 
used to specify an infinite sequence over the 
natural numbers, $n \in \mathbb{N}$, where we define 
$\mathbb{N} = \{0, 1, 2, \ldots\}$ and 
$\mathbb{Z}^{+} = \{1, 2, 3, \ldots\}$. 
The ordinary (OGF) and exponential (EGF) generating functions for a 
prescribed sequence, $\langle f_n \rangle$, are defined by the respective 
formal power series in the notation of 
\begin{align*} 
F_f(z) & := \OGFSeriesTransform{z}{f_0, f_1, f_2, \ldots} \equiv 
     \sum_{n=0}^{\infty} f_n z^n \\ 
\widehat{F}_f(z) & := \EGFSeriesTransform{z}{f_0, f_1, f_2, \ldots} \equiv 
     \sum_{n=0}^{\infty} f_n \frac{z^n}{n!}. 
\end{align*} 

\itemII{Power Series and Coefficient Extraction} 
Given the (formal) power series, $F_f(z) := \sum_{n} f_n z^n$, in $z$ 
that enumerates a sequence taken over $n \in \mathbb{Z}$, the 
notation used to extract the coefficients of the %(formal) 
power series expansion of $F_f(z)$ is defined as 
$f_n \equiv [z^n] F_f(z)$. 

\itemII{Iverson's Convention} 
The notation 
$\Iverson{\mathtt{cond}}$ is used as in the references \citep{GKP} to 
specify the value $1$ if the Boolean--valued input 
\texttt{cond} evaluates to true and $0$ otherwise. 
The notation $\Iverson{n = k = 0}$ is equivalent to 
$\delta_{n,0} \delta_{k,0}$, and where 
$\delta: \mathbb{Z}^{2} \rightarrow \{0,1\}$ is %the 
\keywordemph{Kronecker's delta function} and 
$\delta_{i,j} = 1$ if and only if $i = j$. 

\end{itemize} 

\section{Introduction} 
\label{Section_Introduction} 

\subsection{Motivation} 

Many generating functions and special function series of interest in 
combinatorics and number theory satisfy so-termed 
\emph{unilateral} or \emph{bilateral} ``\emph{square series}'' 
power series expansions of the form 
\begin{align} 
\label{eqn_Fsq_SqSeries_first_series_def-intro_v0.1} 
\widetilde{F}_{\sq}(f; q, z) & := 
     \sum_{n \in \mathbb{S}} f_n q^{n^2} z^n, 
\end{align} 
for some $q \in \mathbb{C}$, a prescribed sequence, 
$\langle f_n \rangle$, whose 
\emphonce{ordinary generating function}, 
is analytic on an open disk, $z \in R_0(f)$,  
and for some indexing set $\mathbb{S} \subseteq \mathbb{Z}$. 
Bilateral square series expansions in the form of 
\eqref{eqn_Fsq_SqSeries_first_series_def-intro_v0.1} 
arise in the series expansions of classical identities for infinite products 
such as in the famous \keywordemph{triple product identity} 
\begin{align} 
\label{eqn_JacobiTripleProd-intro_stmt} 
\sum_{n=-\infty}^{\infty} x^{n^2} z^n & = 
     \prod_{n=1}^{\infty} (1-x^{2n})(1+x^{2n-1} z)(1+x^{2n-1} z^{-1}), 
\end{align} 
and in the \keywordemph{quintuple product identity} given by 
\begin{align} 
\label{eqn_QunitupleProductIdent_stmt-intro_stmt} 
\sum_{n=-\infty}^{\infty} (-1)^n q^{n(3n-1) / 2} z^{3n} (1-z q^n) & = 
     \left(q, -z, -q/z; q\right)_{\infty} \left(q z^2, q / z^2; q^2 
     \right)_{\infty}, 
\end{align} 
for $|q| < 1$, any $z \neq 0$, and where the 
\keywordemph{$q$--Pochhammer symbols}, 
$\QPochhammer{a}{q}{n} \equiv (1-a)(1-aq) \cdots (1-a q^{n-1})$ and 
$\QPochhammer{a_1,a_2,\ldots,a_r}{q}{n} := 
 \prod_{i=1}^{r} \QPochhammer{a_i}{q}{n}$, 
denote the multiple infinite products 
on the right-hand-side of \eqref{eqn_QunitupleProductIdent_stmt-intro_stmt} 
\citep[\S 17.8, \S 20.5]{NISTHB} \citep[\S 19]{HARDYWRIGHTNUMT}. 

The \keywordemph{Jacobi theta functions}, denoted by 
$\vartheta_i(u, q)$ for $i = 1,2,3,4$, or by 
$\vartheta_i\left(u \vert \tau\right)$ when the 
\keywordemph{nome} $q \equiv \exp\left(\imath\pi\tau\right)$ 
satisfies $\Im(\tau) > 0$, 
form another class of square series expansions of 
special interest in the applications considered within the article. 
The classical forms of these theta functions satisfy the 
respective \emph{bilateral} and corresponding asymmetric, \emph{unilateral} 
\emph{Fourier--type} square series expansions given by 
\citep[\S 20.2(i)]{NISTHB} 
\label{def_JacobiThetaFns_VarThetaizq_FourierSeries_exps} 
\label{ex_Asymm-Unilateral-SqSeries_Variants_and_GFs} 
\begin{subequations} 
\begin{alignat}{2} 
\notag 
\vartheta_1(u, q) & = 
     \sum_{n=-\infty}^{\infty} 
     q^{\left(n+\frac{1}{2}\right)^2} (-1)^{n-1/2} e^{(2n+1) \imath u} && = 
     2 q^{1/4} \sum_{n=0}^{\infty} q^{n(n+1)} (-1)^{n} 
     \sin\left((2n+1) u\right) \\ 
\notag 
\vartheta_2(u, q) & = 
     \sum_{n=-\infty}^{\infty} 
     q^{\left(n+\frac{1}{2}\right)^2} e^{(2n+1) \imath u} && = 
     2 q^{1/4} \sum_{n=0}^{\infty} q^{n(n+1)} 
     \cos\left((2n+1) u\right) \\ 
\notag 
\vartheta_3(u, q) & = 
     \sum_{n=-\infty}^{\infty} 
     q^{n^2} e^{2n \imath u} && = 
     1 + 2 \sum_{n=1}^{\infty} q^{n^2} \cos\left(2n u\right) \\ 
\notag 
\vartheta_4(u, q) & = 
     \sum_{n=-\infty}^{\infty} 
     q^{n^2} (-1)^{n} e^{2n \imath u} && = 
     1 + 2 \sum_{n=1}^{\infty} q^{n^2} (-1)^{n} \cos\left(2n u\right). 
\end{alignat} 
\end{subequations} 
Additional square series expansions related to these forms are derived 
from the special cases of the Jacobi theta functions 
defined by $\vartheta_i(q) \equiv \vartheta_i(0, q)$, and by 
$\vartheta_i^{(j)}(u, q) \equiv 
 \partial^{(j)}{\vartheta_i(u_0, q)} / \partial{u_0}^{(j)} \vert_{u_0=u}$ 
or by $\vartheta_i^{(j)}(q) \equiv \vartheta_i^{(j)}(0, q)$ 
where the higher-order $j^{th}$ derivatives are taken over 
any fixed $j \in \mathbb{Z}^{+}$. 

\begin{definition}[Square Series Generating Functions] 
For a given infinite sequence, $\langle f_n \rangle = \{f_0, f_1, f_2, \ldots\}$, 
its ordinary generating function (OGF) is defined as the formal power series 
\begin{align*}
F_f(z) & := \sum_{n=0}^{\infty} f_n z^n, 
\end{align*} 
while its \emph{square series expansion}, or corresponding 
\emph{square series generating function}, is defined by 
\begin{align} 
\label{eqn_Unilateral_Fsq_SqSeries_GF} 
F_{\sq}(f; q, z) & := \sum_{n \geq 0} f_n q^{n^2} z^n, 
\end{align} 
for some $q \in \mathbb{C}$ such that $|q| \leq 1$ and where the ordinary 
generating function defined above is analytic on some non-trivial open disk, 
$R_0(f)$. We note that the definition of the unilateral square series 
generating function defined by \eqref{eqn_Unilateral_Fsq_SqSeries_GF} 
corresponds to taking the indexing set of $\mathbb{S} \equiv \mathbb{N}$ in the 
first definition of \eqref{eqn_Fsq_SqSeries_first_series_def-intro_v0.1}. 
Most of the examples of bilateral series expansions where 
$\mathbb{S} \equiv \mathbb{Z}$ we will encounter as applications 
in the article are reduced to the sum of two unilateral square series 
generating functions of the form in \eqref{eqn_Unilateral_Fsq_SqSeries_GF}. 
\end{definition} 

\subsection{Approach} 
\label{subSection_Intro_Approach} 

We employ a new alternate generating--function--based approach to the 
unilateral square series expansions in 
\eqref{eqn_Unilateral_Fsq_SqSeries_GF} within this article which lead to new 
integral transforms that generate these power series expansions. 
The approach to expanding the forms of the unilateral square series in the 
article mostly follows from results that are rigorously 
justified as operations on formal power series. 
As a result, these transformations provide new approaches to 
these series and other insights that are not, by necessity, 
directly related to the underlying analysis or 
combinatorial interpretations of these special function series. 

The particular applications of the new results and integral representations 
proved within the article typically fall into one of the 
following three primary generalized classes of 
the starting sequences, $\langle f_n \rangle$, 
in the definition of 
\eqref{eqn_Unilateral_Fsq_SqSeries_GF} given above: 
\begin{itemize} 
     \newcommand{\itemmarker}{\tiny{$\blacktriangleright$}} 
     \newcommand{\itemlabel}[1]{\textbf{#1}:} 
     \newcommand{\itemII}[1]{\item[\itemmarker] \itemlabel{#1}} 

\itemII{Geometric-Series-Based Sequence Functions (Section \ref{Section_Applications_of_GeomSqSeries})} 
Variations of the so-termed \emph{geometric square series functions} 
over sequences of the form $f_n \equiv p(n) \cdot c^n$ 
with $c \in \mathbb{C}$ and a fixed 
$p(n) \in \mathbb{C} \lbrack n \rbrack$, 
including the special cases where $p(n) := 1, (an+b), (an+b)^{m}$ 
for some fixed constants, $a, b \in \mathbb{C}$, and 
$m \in \mathbb{Z}^{+}$; 

\itemII{Exponential-Series-Based Generating Functions (Section \ref{Section_AppsOf_ExpSqSeries})} 
The \emph{exponential square series functions} 
involving polynomial multiples of the 
exponential-series-based sequences, 
$f_n \equiv r^n / n!$, for a fixed $r \in \mathbb{C}$; and 

\itemII{Fourier-Type Sequences and Generating Functions (Section \ref{Section_ExpansionsOf_Fourier-TypeSqSeries})} 
\emph{Fourier-type square series functions} which involve the sequences 
$f_n \equiv \scFn(\alpha n+\beta) \cdot c^n$ 
with $\alpha, \beta \in \mathbb{R}$, 
for some $c \in \mathbb{C}$, and 
where the trigonometric function $\scFn \in \{\sin, \cos$\}. 
\end{itemize} 
The special expansions of the square series functions defined in 
Table \ref{table_Notation_SeriesDefsOfSpecialSqSeries-stmts_v1} 
summarize the generalized forms of the applications cited in 
Section \ref{subSection_Intro_Examples} below and in the application 
sections \ref{Section_Applications_of_GeomSqSeries} -- 
\ref{Section_ExpansionsOf_Fourier-TypeSqSeries} of the article. 

\subsection{Terminology} 

We say that a function with a series expansion of the form in  
\eqref{eqn_Fsq_SqSeries_first_series_def-intro_v0.1} or in 
\eqref{eqn_Unilateral_Fsq_SqSeries_GF} 
has a so-termed \emph{square series} expansion of the 
form studied by the article. 
Within the context of this article, the terms 
``\emph{square series}'' and ``\emph{square series integral},'' 
or ``\emph{square series integral representation}'', 
refer to the specific constructions derived from the 
particular generating function transformation results 
established by this article. 
The usage of these terms is then applied interchangeably in the context of 
any number of applications to the many existing well-known 
classical identities, special theta functions, and 
other power series expansions of special functions, 
which are then studied through the new forms of the 
square series transformations and 
integral representations proved within the article. 

\begin{table}[h] 

\begin{center} 

\setlength{\fboxrule}{1.25pt} 
\setlength{\fboxsep}{4pt} 
\fbox{\begin{minipage}{\textwidth} 

     \newcommand{\SqSeriesFnTableSubHeader}[1]{
          \taburowcolors 2{Plum!32!White .. Plum!10!White} 
          \hline 
          \multicolumn4{|l|}{\bfseries\normalsize{#1}} \\ \hline\hline 
          \rowfont{\bfseries\small} Eq. & Function & Series & Parameters \\ \hline%\hline 
          \taburowcolors 2{Plum!0!White .. Plum!0!White} 
     } 
     \newcommand{\AddTableHeader}[1]{\SqSeriesFnTableSubHeader{#1}} 

     \tabulinestyle{2pt OTblBC} 
     \taburulecolor |OTblBC|{ITblBC} 
     \arrayrulewidth=1.5pt \doublerulesep=0pt \tabulinesep=3pt%\extrarowsep=3pt
     \begin{longtabu} to \TblW {|c|c|T{+0}{l}|X[l]|} %\hline 
        \AddTableHeader{Geometric Square Series Functions} 

        \LabelTableEqn{eqn_Ordinary_and_ExpGeomSqSeries_defs_stmts_v2} & 
           $G_{\sq}(q, c, z)$ & 
           $\sum_{n=0}^{\infty} q^{n^2} c^n z^n$ & 
           $|cz| < 1$ \\ 
        \hline 
       
        \LabelTableEqn{eqn_Theta_cqz_GeomSeries-Based_series_stmt} & 
           $\vartheta_d(q, c, z)$ & 
           $\sum_{n=d}^{\infty} q^{n^2} c^n z^n$ & 
           $|cz| < 1; d \in \mathbb{Z}^{+}$ \\ 
        \hline 

       \LabelTableEqn{eqn_GeomSqSeries_Qabqcz-intro_series_exp_v1} & 
           $Q_{a,b}(q, c, z)$ & 
           $\sum_{n=0}^{\infty} q^{n^2} (an+b) c^n z^n$ & 
           $|cz| < 1; a,b \in \mathbb{C}$ \\ 
       \hline 

       \LabelTableEqn{eqn_Theta_dm_cqz_ShiftedPolyPowers_GeomSeries-Based_series_stmt_v1} & 
           $\vartheta_{d,m}(q, c, z)$ & 
           $\sum_{n=d}^{\infty} n^{m} \times q^{n^2} c^n z^n$ & 
           $|cz| < 1; d,m \in \mathbb{N}$ \\ 
       \hline 

       \LabelTableEqn{eqn_Theta_dm_cqz_ShiftedPolyMultiples_GeomSeries-Based_series_stmt_v2} & 
           $\vartheta_{d,m}(\alpha, \beta; q, c, z)$ & 
           $\sum_{n=d}^{\infty} \left(\alpha n+\beta\right)^{m} 
            \times q^{n^2} c^n z^n$ & 
           ${|cz| < 1}$; ${\alpha,\beta \in \mathbb{C}}$; 
           ${d,m \in \mathbb{N}}$ \\ 
       \hline 

       \LabelTableEqn{eqn_GeomSqSeries_SpFnVariants_GsqpmqcFn_series_def_v1} & 
           $G_{\sq}(p, m; q, c)$ & 
           $\sum_{n=0}^{\infty} \left(q^{p}\right)^{n^2} 
            \left(c \cdot q^{m}\right)^{n}$ & 
           $|q^{m} cz| < 1$; 
           $p,m \in \mathbb{R}^{+}$ \\ 
       \hline 
     
       \AddTableHeader{Fourier--Type Square Series Functions} 

       \LabelTableEqn{eqn_FscAlphaBetauqz_FourierSquareSeries-introc_def_v3} & 
           $F_{\scFn}(\alpha, \beta; q, c, z)$ & 
           $\sum_{n=0}^{\infty} q^{n^2} 
            \scFn\left(\alpha n+\beta\right) c^n z^n$ & 
           ${|cz| < 1}$ \\ %; 
       \hline 

       \LabelTableEqn{eqn_JacobiTheta_fn_T.1} & 
           $\vartheta_1(u, q, z)$ & 
           $\scriptstyle{2 q^{1/4} \times \sum_{n=0}^{\infty} 
            q^{n^2} (-1)^{n} \sin\left((2n+1) u\right) q^{n} z^n}$ & 
           $|q^{} z| < 1$; 
           $u \in \mathbb{R}$ \\ 
       \hline%\hline 
 
       \LabelTableEqn{eqn_JacobiTheta_fn_T.2} & 
           $\vartheta_2(u, q, z)$ & 
           $\scriptstyle{2 q^{1/4} \times \sum_{n=0}^{\infty} 
            q^{n^2} \cos\left((2n+1) u\right) q^{n} z^n}$ & 
           $|q^{} z| < 1$; 
           $u \in \mathbb{R}$ \\ 
       \hline 

       \LabelTableEqn{eqn_JacobiTheta_fn_T.3} & 
           $\vartheta_3(u, q, z)$ & 
           $\scriptstyle{1 + 2q \times \sum_{n=0}^{\infty} 
            q^{n^2} \cos\left((2n+2) u\right) q^{2n} z^{n}}$ & 
           $|q^{2} z| < 1$; 
           $u \in \mathbb{R}$ \\ 
       \hline 

       \LabelTableEqn{eqn_JacobiTheta_fn_T.4} & 
           $\vartheta_4(u, q, z)$ & 
           $\scriptstyle{1 - 2q \times \sum_{n=0}^{\infty} 
            q^{n^2} (-1)^{n} \cos\left((2n+2) u\right) q^{2n} z^{n}}$ & 
           $|q^{2} z| < 1$; 
           $u \in \mathbb{R}$ \\ 
       \hline 

       \AddTableHeader{Exponential Square Series Functions} 

       \LabelTableEqn{eqn_Esq_qrz_series_def_v1} & 
           $E_{\sq}(q, r, z)$ & 
           $\sum_{n=0}^{\infty} q^{n^2} r^n z^n / n!$ & 
           $r,z \in \mathbb{C}$ \\ 
       \hline 

       \LabelTableEqn{eqn_ETildesq_qBinomPow_rz_int_rep_series_def_v1} & 
           $\widetilde{E}_{\sq}(q, r, z)$ & 
           $\sum_{n=0}^{\infty} q^{\binom{n}{2}} r^n z^n / n!$ & 
           $r,z \in \mathbb{C}$ \\ 
       \hline 

    \end{longtabu}

     \setcaption{table}{Special Classes of Generalized Square Series Functions}{ 
                   Definitions of the generalized series expansions for the 
                   forms of several parameterized classes of 
                   special square series functions considered by the article.} 
     \label{table_Notation_SeriesDefsOfSpecialSqSeries-stmts_v1} 

\end{minipage}}
\end{center} 
\end{table} 

\subsection{Examples and Applications} 
\label{subSection_Intro_Examples} 

We first briefly motivate the utility to the sequence--based 
generating function approach implicit to the square series expansions 
suggested by the definition of 
\eqref{eqn_Unilateral_Fsq_SqSeries_GF} 
through several examples of the new integral representations 
following as consequences of the new results established in 
Section \ref{Section_SquareSeriesGF_transforms}. 
Sections \ref{Section_Applications_of_GeomSqSeries} -- 
\ref{Section_ExpansionsOf_Fourier-TypeSqSeries} of the article
also consider integral representations for generalized forms of the 
three primary classes of square series expansions listed in 
Section \ref{subSection_Intro_Approach}, 
as well as other concrete examples of concrete special case applications 
of our new results. 
This subsection is placed in the introduction both to 
demonstrate the stylistic natures of the new square series integral 
representations we prove in later sections of the article and to 
motivate Theorem \ref{thm_SqSeries_OGF_Transforms} by 
demonstrating the breadth of applications to special functions offered by 
our new results. 

\subsubsection*{Series for the Infinite $q$-Pochhammer Symbol} 

The next results illustrate several different series forms 
for \keywordemph{Euler's function}, $f(q) \equiv \QP{q}$, 
given by the respective bilateral and unilateral series for the 
infinite product expansions given by 
\citep[\S 14.4]{APOSTOLANUMT} \citep[\S 27.14(ii)]{NISTHB} 
\begin{align} 
\label{eqn_qSeries_InfProduct_to_Series_stmt} 
(q)_{\infty} & = \prod_{n=1}^{\infty} \left(1 - q^n\right) = 
     1 -q - q^2 + q^5 + q^7 - q^{12} - q^{15} + \cdots \\ 
\notag 
   & = 
     \sum_{n=-\infty}^{\infty} (-1)^{n} q^{n(3n+1) / 2} = 
     1 + \sum_{n=1}^{\infty} (-1)^n \left(q^{\omega(n)} + 
     q^{\omega(-n)}\right), 
\end{align} 
for $|q| < 1$ and 
where the series exponents, $\omega(\pm k) = (3 k^2 \mp k) / 2$, 
in the last equation denote the 
\keywordemph{pentagonal numbers} \citep[A000326]{OEIS}. 
The cube powers of this product are expanded by 
\keywordemph{Jacobi's identity} as the following series whenever $|q| < 1$ 
\citep[Thm. 1.3.9]{NUMTSPRAM} 
\citep[Thm. 357; \S 19]{HARDYWRIGHTNUMT} 
\citep[\cf Entry 22]{BERNDTRAMNBIII} 
\citep[\cf \Section 17.2(i)]{NISTHB}: 
\begin{align} 
\label{eqn_qSeries_qqinf_Pow3} 
\QP{q}^{3} & = \prod_{n=1}^{\infty} (1-q^n)^3 = 
     1 + \sum_{m=1}^{\infty} (-1)^m\ (2m+1)\ q^{m(m+1) / 2}. 
\end{align} 
For a sufficiently small real--valued $q$ or %a 
strictly complex-valued $q$ such that $\Im(q) > 0$, each with 
$|q| \in (0, 1]$, these products satisfy the next series expansions 
formed by the special cases of 
\eqref{eqn_Unilateral_Fsq_SqSeries_GF} given by 
\citep[\S 23.17(iii), \S 27.14(ii)--(iv)]{NISTHB} 
\citep[\S 19.9]{HARDYWRIGHTNUMT} 
\begin{subequations} 
\begin{align}
\label{eqn_PartitionFn_pn_OGF_series_and_reciprocal_IntRep-stmts_v0} 
\QP{q} & = 
     1 - q \times \sum_{n=0}^{\infty} (-1)^{n} q^{n(3n+5)/2} + 
     \sum_{n=1}^{\infty} (-1)^{n} q^{n(3n+1) / 2} \\ 
\notag 
   & = 
     1 - q \times F_{\sq}\left(f_{11}; q^{3/2}, -1\right) - 
     q^{2} \times F_{\sq}\left(f_{12}; q^{7/2}, -1\right) \\ 
\notag 
     & = 
     1 - \int_0^{\infty} \frac{2 e^{-t^2/2}}{\sqrt{2\pi}} \left[ 
     \frac{q \left(1 + q^{5/2} \cosh\left(\sqrt{3 \Log(q)} t\right) 
     \right)}{q^5 + 2 q^{5/2} \cosh\left(\sqrt{3 \Log(q)} t\right) + 1}
     \right] dt \\ 
\notag 
   & \phantom{= 1\ } - 
     \int_0^{\infty} \frac{2 e^{-t^2/2}}{\sqrt{2\pi}} \left[ 
     \frac{q^2 \left(1 + q^{7/2} \cosh\left(\sqrt{3 \Log(q)} t\right) 
     \right)}{q^7 + 2 q^{7/2} \cosh\left(\sqrt{3 \Log(q)} t\right) + 1}
     \right] dt, |q| \in (0, 1/9) \\ 
\label{eqn_PartitionFn_pn_OGF_series_and_reciprocal_VariousIntReps_for_QPqFn} 
\QP{q}^{3} & = \frac{1}{2 q^{1/8}} 
     \vartheta_1^{\prime}\left(\sqrt{q}\right) = 
     1 + \sum_{n=1}^{\infty} q^{n(n+1)/2} (2n+1) (-1)^{n} \\ 
\notag 
   & = 
     1 - q \times F_{\sq}\left(f_2; q^{1/2}, -1\right) \\ 
\notag 
     & = 
     1 + \int_0^{\infty} \frac{e^{-t^2/2}}{\sqrt{2\pi}} \left[ 
     \frac{4 q^{5/2} \left((q^3+1) \cosh\left(\sqrt{\Log(q)} t\right) + 
     2 q^{3/2}\right)}{
     \left(q^3 + 2 q^{3/2} \cosh\left(\sqrt{\Log(q)} t\right) + 1\right)^2}
     \right] dt \\ 
\notag 
   & \phantom{= 1\ } - 
     \int_0^{\infty} \frac{e^{-t^2/2}}{\sqrt{2\pi}} \left[ 
     \frac{6q \left(1 + q^{3/2} \cosh\left(\sqrt{\Log(q)} t\right) 
     \right)}{q^3 + 2 q^{3/2} \cosh\left(\sqrt{\Log(q)} t\right) + 1}
     \right] dt,\ 
     |q| \in \left(0, 2^{-2/3}\right) \\ 
\QP{q} & = \frac{1}{3^{1/2} q^{1/24}} \times 
     \vartheta_2\left(\pi / 6, q^{1/6}\right) \\ 
\notag 
   & = 
     \frac{2 \sqrt{3}}{3} \times \sum_{n=0}^{\infty} 
     q^{n(n+1)/6} (-1)^{n} \cos\left((2n+1) \cdot \frac{\pi}{6}\right) \\ 
\notag 
   & = 
     \frac{2 \sqrt{3}}{3} \times 
     F_{\sq}\left(f_3; q^{1/6}, -1\right) \\ 
\notag 
     & = 
     \int_0^{\infty} \frac{2 e^{-t^2 / 2}}{\sqrt{2\pi}} 
     \mathsmaller{
     \frac{\sqrt{q} \cosh\left(\frac{\sqrt{\Log(q)} t}{\sqrt{3}}\right) z^3 + 
     2 q^{1/3} \cosh\left(\frac{\sqrt{\Log(q)} t}{\sqrt{3}}\right)^2 z^2 + 
     2 q^{1/6} \cosh\left(\frac{\sqrt{\Log(q)} t}{\sqrt{3}}\right) z + 1}{ 
     q^{2/3} z^4 + 2 \sqrt{q} 
     \cosh\left(\frac{\sqrt{\Log(q)} t}{\sqrt{3}}\right) z^3 + 
     q^{1/3} \left(1 + 2 \cosh\left(\frac{\sqrt{\Log(q)} t}{\sqrt{3}}\right) 
     \right) z^2 + 2 q^{1/6} 
     \cosh\left(\frac{\sqrt{\Log(q)} t}{\sqrt{3}}\right) z + 1} 
     } dt, 
\end{align} 
\end{subequations} 
when the respective sequence forms from the square series functions in the 
previous equations are defined to be 
$f_{11}(n) := q^{5n/2}$, $f_{12}(n) := q^{7n/2}$, 
$f_2(n) := (2n+3) \cdot q^{3n/2}$, and 
$f_3(n) := q^{n/6} \cdot \cos\left((2n+1) \cdot \pi / 6\right)$. 

\subsubsection*{The General Two-Variable Ramanujan Theta Function} 
\label{example_2Var_RamThetaFn_fab} 
\label{cor_2Var_RamThetaFn_fab} 

For any non-zero $a,b \in \mathbb{C}$ with $|ab| < 1$, the 
\keywordemph{general (two-variable) Ramanujan theta function}, $f(a, b)$, 
is defined by the bilateral series expansion 
\citep[\Section 16, (18.1)]{BERNDTRAMNBIII} 
\citep[\Section 20.11(ii)]{NISTHB} 
\begin{align} 
\label{eqn_RamThetaFn_fab_bilateral_series_exps_stmt} 
f(a, b) & := \sum_{n=-\infty}^{\infty} a^{n(n+1) / 2} b^{n(n-1) / 2} \\ 
\notag 
   & \phantom{:} = 
     1 + (a+b) + ab(a^2+b^2) + a^3b^3(a^3+b^3) + a^6b^6(a^4+b^4)+ \cdots, 
\end{align} 
where the second infinite expansion of the series for $f(a,b)$ 
corresponds to the leading powers of $ab$ taken over the 
\keywordemph{triangular numbers}, $T_n \equiv n(n-1)/2$, 
\citep[A000217]{OEIS} \citep[\S XII]{RAMANUJAN}. 
The general Ramanujan theta function 
is also expanded by the pair of unilateral 
geometric square series functions from 
Table \ref{table_Notation_SeriesDefsOfSpecialSqSeries-stmts_v1} as 
\begin{align} 
\label{eqn_RamThetaFn_fab_unilateral_series_exp_stmts-series_v2} 
f(a, b) & = 1 + \sum_{n=1}^{\infty} \left[a^{n(n+1) / 2} b^{n(n-1) / 2} + 
     a^{n(n-1) / 2} b^{n(n+1) / 2}\right] \\ 
\notag 
   & = 
     1 + \vartheta_1\left(\sqrt{ab}, \sqrt{a b^{-1}}, 1\right) + 
     \vartheta_1\left(\sqrt{ab}, \sqrt{b a^{-1}}, 1\right) \\ 
\label{eqn_RamThetaFn_fab_unilateral_series_exp_stmts-series_v3} 
f(a, b) & = 
     1 + \sum_{n=0}^{\infty} \left[ 
     a \times a^{n(n+3)/2} b^{n(n+1)/2} + 
     b \times a^{n(n+1)/2} b^{n(n+3)/2} 
     \right] \\ 
\notag 
   & = 
     1 + 
     a \times G_{\sq}\left(\sqrt{ab}, a \sqrt{ab}, 1\right) + 
     b \times G_{\sq}\left(\sqrt{ab}, b \sqrt{ab}, 1\right). 
\end{align} 
These symmetric forms of the unilateral series 
satisfy the well known property for the function that 
$f(a, b) \equiv f(b, a)$ 
\citep[\Section 1]{NUMTSPRAM} \citep[\cf \Section 5]{RRIDENTLIST} and 
lead to the new integral representations in the next formula given by 
\begin{align} 
\label{eqn_RamThetaFn_fab_unilateral_series_exp_stmts-int_rep_formula_v5} 
f(a,b) & = 1 + \int_0^{\infty} \frac{2a e^{-t^2/2}}{\sqrt{2\pi}}\left[ 
     \frac{1 - a \sqrt{ab} \cosh\left(\sqrt{\Log(ab)} t\right)}{ 
     a^3 b - 2a \sqrt{ab} \cosh\left(\sqrt{\Log(ab)} t\right) + 1} 
     \right] dt \\ 
\notag 
   & \phantom{=\ 1 } + 
     \int_0^{\infty} \frac{2b e^{-t^2/2}}{\sqrt{2\pi}}\left[ 
     \frac{1 - b \sqrt{ab} \cosh\left(\sqrt{\Log(ab)} t\right)}{ 
     a b^3 - 2b \sqrt{ab} \cosh\left(\sqrt{\Log(ab)} t\right) + 1} 
     \right] dt. 
\end{align} 
The full two-parameter form of the Ramanujan theta function, $f(a, b)$, 
defines the expansions for many other special function series. 
The particular notable forms of these modified theta function series 
include the \emph{one--variable}, or \emph{single--argument}, 
form of the function, $f(q) := f(-q, -q^2) \equiv \QP{q}$, 
which is related to the 
\keywordemph{Dedekind eta function}, 
$\eta(\tau) \equiv q^{1/12} \QP{\bar{q}}$ when 
$\bar{q} \equiv \exp\left(2\pi\imath \tau\right)$ 
is the square of the \keywordemph{nome} $q$, and the 
special cases of \keywordemph{Ramanujan's functions}, 
$\varphi(q) \equiv f(q, q)$ and $\psi(q) \equiv f(q, q^3)$, 
considered in Section \ref{subsubSection_Apps_SpFnsRelated_to_GeomSqSeries_RamPsiPhiFns} 
\citep[\S 20.11(ii); \S 27.14]{NISTHB}. 
\citep[\Section XII]{RAMANUJAN}. 
The definition of the 
general two-variable form of the theta function, $f(a, b)$, 
is also given in terms of the classical \emph{Jacobi theta function}, 
$\vartheta_3(z, q) \equiv \vartheta_3(z\vert\tau)$, 
where $a := q e^{2 \imath z}$, $b := q e^{-2 \imath z}$, and 
$q \equiv e^{\imath\pi\tau}$ 
\citep[\S 20.11(ii)]{NISTHB}, and alternately as 
$\vartheta_3(v, \tau)$ over the inputs of 
$v := \log\left(\frac{a}{b}\right) \times \left(4\pi\imath\right)^{-1}$ and 
$\tau := \log\left(ab\right) \times \left(2\pi\imath\right)^{-1}$ 
in \citep[\S XII]{RAMANUJAN}. 

\subsubsection*{New Integral Formulas for Explicit Values of Special Functions} 

Similar series related to the Jacobi theta functions, 
$\vartheta_i(0, q)$, 
lead to further new, exact integral formulas for certain 
explicit special constant values involving the 
\keywordemph{gamma function}, $\Gamma(z)$, at the rational inputs of 
$\Gamma(1/4)$, $\Gamma(1/3)$, and $\Gamma(3/4)$. 
Specific examples of the new integral formulas derived from the 
results proved within this article include the following 
special cases of the results for Ramanujan's functions given in 
\sref{Section_Applications_of_GeomSqSeries}: 
\begin{align} 
\notag 
\varphi\left(e^{-5\pi}\right) & \equiv 
\frac{\pi^{1/4}}{\Gamma\left(\frac{3}{4}\right)} \cdot 
     \frac{\sqrt{5 + 2 \sqrt{5}}}{5^{3/4}} \\ 
\notag 
   & = 
     1 + \int_0^{\infty} \frac{e^{-t^2/2}}{\sqrt{2\pi}} \left[ 
     \frac{4 e^{5\pi} \left(e^{10\pi} - \cos\left(\sqrt{10 \pi} t\right) 
     \right)}{e^{20\pi} - 2 e^{10\pi} \cos\left(\sqrt{10 \pi} t\right) + 1} 
     \right] dt \\ 
\notag 
\psi\left(e^{-\pi / 2}\right) & \equiv 
     \frac{\pi^{1/4}}{\Gamma\left(\frac{3}{4}\right)} \cdot 
     \frac{\left(\sqrt{2} + 1\right)^{1/4} e^{\pi / 16}}{2^{7/16}} \\ 
\notag 
   & = 
     \int_0^{\infty} \frac{e^{-t^2/2}}{\sqrt{2\pi}} \left[ 
     \frac{\cos\left(\sqrt{\frac{\pi}{2}} t\right) - e^{\pi / 4}}{ 
     \cos\left(\sqrt{\frac{\pi}{2}} t\right) - \cosh\left(\frac{\pi}{4}\right)} 
     \right] dt. 
\end{align} 
When $q := \pm \exp\left(-k\pi\right)$ for some real-valued $k > 0$, the 
forms of the Ramanujan theta functions, $\varphi(q)$ and $\psi(q)$, 
define the entire classes of special constants studied in 
\citep[\cf \S 5]{THETAFN-IDENTS-EXPLICIT-FORMULAS} 
\citep[\cf \S 5--6]{RAMTHETAFNS-EXPLICIT-VALUES}. 
We also see that similar expansions can be given for constant values of the 
\keywordemph{Dedekind eta function}, $\eta(\tau)$, through it's 
relations to the Jacobi theta functions and to the q-Pochhammer symbol 
\citep[\S 23.15(ii), \S 23.17(i), \S 27.14(iv)]{NISTHB}. 

\subsubsection*{Exponential Generating Functions} 

The examples cited so far in the introduction have involved 
special cases of the geometric square series expansions, or the 
sequences over Fourier-type square series, which are readily 
expanded through the new integral representation formulas for these 
geometric-series-based sequence results. 
Before concluding this subsection, we suggest additional, 
characteristically distinct applications 
that are derived from exponential-series-based sequences and 
generating functions. 
The resulting integral representations derived from 
Theorem \ref{thm_SqSeries_OGF_Transforms} 
for these exponential square series types represent a 
much different stylistic nature 
for these expansions than for the analogous non--exponential square series 
generating functions demonstrated so far in the examples above 
(\cf the characteristic expansions cited in 
Section \ref{subSection_EGFs_AComparison}). 

In particular, 
we consider the particular applications of two--variable sequence 
generating functions over the \emph{binomial powers} of a 
fixed series parameter, $q^{\binom{n}{2}} \equiv q^{n(n-1)/2}$, 
defined by the series 
%% See : "temp-working-2013.03.26-v1.*": 
\begin{align} 
\label{eqn_Eqsqz_series_exp_and_IntRep-intro_example_v1} 
\widetilde{E}_{\sq}(q, z) & := 
     \sum_{n=0}^{\infty} q^{\binom{n}{2}} \frac{z^n}{n!} = 
     F_{\sq}\left(\left\langle q^{-n/2}/n! \right\rangle; q^{1/2}, z\right) \\ 
\notag 
   & \phantom{:} = 
     \int_0^{\infty} \frac{e^{-t^2/2}}{\sqrt{2\pi}} \left[ 
     \sum_{b=\pm 1} 
     \exp\left(e^{b t \sqrt{\Log(q)}} \cdot \frac{z}{\sqrt{q}}\right) 
     \right] dt. 
\end{align} 
For $m,n \in \mathbb{N}$, let the function $\ell_m(n)$ denote the 
\keywordemph{number of labeled graphs with $m$ edges on $n$ nodes}. 
The sequence is generated as the coefficients of the power series expansion 
formed by the special case of the two--variable, double generating function in 
\eqref{eqn_Eqsqz_series_exp_and_IntRep-intro_example_v1} given by 
\citep[\S 2]{BRIGGS-ENUMLGRAPHS} 
%% See : "temp-working-2014.01.31-v1.*": 
\label{example_EGF_LabeledGraphs_gwz_ellmn} 
\begin{align} 
\label{eqn_gwz_LabeledGraph_SqSeries_EGF_result} 
\widehat{G}_{\ell}(w, z) & := 
     \sum_{n,m \geq 0} \ell_m(n) w^m \frac{z^n}{n!} = 
     \sum_{n=0}^{\infty} (1+w)^{\binom{n}{2}} \frac{z^n}{n!} \\ 
\notag 
   & \phantom{:} = 
     \int_0^{\infty} \frac{e^{-t^2/2}}{\sqrt{2\pi}} \left[ 
     e^{e^{\sqrt{\Log(1+w)} t} \frac{z}{\sqrt{1+w}}} + 
     e^{e^{-\sqrt{\Log(1+w)} t} \frac{z}{\sqrt{1+w}}} 
     \right] dt. 
\end{align} 
Other examples of series related to the definition in 
\eqref{eqn_Eqsqz_series_exp_and_IntRep-intro_example_v1} 
arise in the combinatorial interpretations and generating functions over 
sequences with applications to graph theoretic contexts. 
The results proved in \sref{subSection_EGFs_and_GraphTheory_Examples} 
provide further applications of the new integral representations for these 
variants of the exponential square series. 

\subsubsection*{Other Applications of the New Results in the Article} 

There are a number of other applications of the new results and 
square series integral representations proved within the article. 
The next few identities provide additional examples and 
specific applications of these new results. 
A pair of 
$q$-series expansions related to \emph{Zagier's identities} and the 
second of the infinite products defined in 
\eqref{eqn_qSeries_InfProduct_to_Series_stmt} 
are stated by the next equations for $|q| < 1$ 
\citep[\S 3; Thm. 1]{CHAPMAN-IDENTOFZAGIER}. 
\begin{align} 
\label{eqn_ZagiersIdentity_first_series_exp_v1} 
\sum_{n=0}^{\infty} \QPochhammer{z}{q}{n+1} z^n & = 1 + \sum_{n=1}^{\infty} 
     (-1)^n \left[q^{n(3n-1) / 2} z^{3n-1} + q^{n(3n+1) / 2} z^{3n}\right] \\ 
\notag 
     & = 
     1 - q z^2 \int_0^\infty \frac{2 e^{-t^2 / 2}}{\sqrt{2\pi}} 
     \frac{\left(1 + q^{5/2} \cosh\left(\sqrt{3 \Log(q)} t\right) \right)}{ 
     \left(q^5 z^6 + 2 q^{5/2} z^3 
     \cosh\left(\sqrt{3 \Log(q)} t\right) + 1\right)} dt \\ 
\notag 
     & \phantom{= 1\ } - 
     q^2 z^3 \int_0^\infty \frac{2 e^{-t^2 / 2}}{\sqrt{2\pi}} 
     \frac{\left(1 + q^{7/2} \cosh\left(\sqrt{3 \Log(q)} t\right) \right)}{ 
     \left(q^7 z^6 + 2 q^{7/2} z^3 
     \cosh\left(\sqrt{3 \Log(q)} t\right) + 1\right)} dt \\ 
\notag 
\sum_{n=0}^{\infty} \left[(q)_{\infty} - (q)_n\right] & = 
     (q)_{\infty} \sum_{n=1}^{\infty} \frac{q^n}{1 - q^n} + 
     \sum_{n=1}^{\infty} (-1)^{n} \left[(3n-1) q^{n(3n-1) / 2} + 
     (3n) q^{n(3n+1) / 2}\right] \\ 
\notag 
     & = 
     (q)_{\infty} \sum_{n=1}^{\infty} \frac{q^n}{1 - q^n} \\ 
\notag 
     & \phantom{=\qquad\ } - 
     \int_0^{\infty} \frac{e^{-t^2 / 2}}{\sqrt{2\pi}} 
     \frac{4q \left(1 + q^{5/2} \cosh\left(\sqrt{3 \Log(q)} t\right)\right)}{ 
     \left(q^5 + 2 q^{5/2} \cosh\left(\sqrt{3 \Log(q)} t\right) + 1 
     \right)} dt \\ 
\notag 
     & \phantom{=\qquad\ } + 
     \int_0^{\infty} \frac{6 q^{7/2} e^{-t^2 / 2}}{\sqrt{2\pi}} 
     \frac{\left(2 q^{5/2} + (1 + q^5)\cosh\left(\sqrt{3 \Log(q)} t\right) 
     \right)}{\left(q^5 + 2 q^{5/2} \cosh\left(\sqrt{3 \Log(q)} t\right) + 1 
     \right)^2} dt \\ 
\notag 
     & \phantom{=\qquad\ } - 
     \int_0^{\infty} \frac{e^{-t^2 / 2}}{\sqrt{2\pi}} 
     \frac{6 q^2 \left(1 + q^{7/2} \cosh\left(\sqrt{3 \Log(q)} t\right)\right)}{ 
     \left(q^7 + 2 q^{7/2} \cosh\left(\sqrt{3 \Log(q)} t\right) + 1 
     \right)} dt \\ 
\notag 
     & \phantom{=\qquad\ } + 
     \int_0^{\infty} \frac{6 q^{11/2} e^{-t^2 / 2}}{\sqrt{2\pi}} 
     \frac{\left(2 q^{7/2} + (1 + q^7)\cosh\left(\sqrt{3 \Log(q)} t\right) 
     \right)}{\left(q^7 + 2 q^{7/2} \cosh\left(\sqrt{3 \Log(q)} t\right) + 1 
     \right)^2} dt 
\end{align} 
Other identities connecting sums involving 
products of terms in $q$-series are found in a number of 
series from Ramanujan's ``lost'' notebook, including many 
expansions that are adapted easily from the mock theta function series 
from Watson's article \citep{WATSONMOCKTHETA,INTRORAMLOSTNB}
\citep[\cf \Section 3 and \Section 4]{BERNDTCOMBIDENTS}. 
Many similar and related identities are also compared in the references 
\citep{ANDREWSQ-SERIES,BERNDTRAMNBIII,ANDREWS2THMS,RRIDENTLIST}. 

Another example utilizing these square series expansions 
provides a new integral representation for an asymptotic formula for the 
\emph{$n^{th}$ prime number}, $p_n$, or alternately, 
a new functional equation for the ordinary generating function of the primes. 
In particular, if we let the constant $K$ be defined by the infinite series 
\begin{equation*}
K = \sum_{k \geq 0} 10^{-k^2} p_k \approx 0.200300005, 
\end{equation*} 
we have the following identities for $p_n$ in terms of its 
ordinary generating function, $\widetilde{P}(z)$, defined such that 
$\widetilde{P}(0) = 0$ whenever $|z| < 1 / 2$ \citep[Ex. 4.21, p. 516]{GKP}: 
\begin{align*} 
p_n  & = \left\lfloor 10^{n^2} K\ \bmod{10^n} \right\rfloor \\ 
     & = 
     \mathsmaller{ 
     \left\lfloor [z^n] 
     \int_0^{\infty} \frac{e^{-t^2 / 2}}{\sqrt{2\pi}} 
     \frac{1 - z \cosh\left(\sqrt{\Log(100)} t\right) dt}{ 
     z^2 - 2 z \cosh\left(\sqrt{\Log(100)} t\right) + 1} 
     \int_0^{\infty} \frac{e^{-s^2/2}}{\sqrt{2\pi}} 
     \mathlarger{\sum_{b = \pm 1}} \widetilde{P}\left(
     \frac{1}{1 - e^{b\imath \sqrt{\Log(100)} t}} 
     \right) ds \bmod{10^n} 
     \right\rfloor. 
     } 
\end{align*} 

\subsubsection*{Unilateral Series Expansions of Bilateral Square Series} 

For fixed parameters $a, b, r_0, r_1, r_2 \in \mathbb{Q}$ and 
$q \in \mathbb{C}$ such that $|q| < 1$, 
let the function, $B_{a,b}\left(r_2, r_1, r_0; q\right)$, be defined by the 
next form of the bilateral series in 
\eqref{eqn_ParamThetaFn_Trsabcq_BilateralSeries_def_stmt-intro_v1}. 
The first series for this function 
is then expanded in the form of a second unilateral square series as 
\begin{align} 
\label{eqn_ParamThetaFn_Trsabcq_BilateralSeries_def_stmt-intro_v1} 
 & B_{a,b}\left(r_2, r_1, r_0; q\right) := \sum_{n=-\infty}^{\infty} 
     (-1)^{n} (an+b) q^{\left(r_2 n^2 + r_1 n + r_0\right) / 2} \\ 
\notag 
   & \phantom{\qquad\qquad} = 
     q^{r_0/2} \left[ 
     \sum_{n=0}^{\infty} (-1)^{n} (an+b) q^{n(r_2 n + r_1) / 2} - 
     \sum_{n=1}^{\infty} (-1)^{n} (an-b) q^{n(r_2 n - r_1) / 2} 
     \right]. 
\end{align} 
An application of 
Proposition \ref{prop_Qsq_abcqz_unilateral_series_fn_integral_rep_v1} 
requires that the second series term in the previous equation be 
shifted to obtain the following unilateral series expansion 
(see Section \ref{Section_Applications_of_GeomSqSeries}): 
\begin{align} 
\label{eqn_ParamThetaFn_Trsabcq_BilateralSeries_def_stmt-intro_v3} 
B_{a,b}\left(r_2, r_1, r_0; q\right) & = 
     q^{r_0/2} \times 
     \sum_{n=0}^{\infty} (-1)^{n} (an+b) q^{n(r_2 n + r_1) / 2} \\ 
\notag 
   & \phantom{=} + 
     q^{(r_2-r_1+r_0)/2} \times 
     \sum_{n=0}^{\infty} (-1)^{n} (an+a-b) q^{n(r_2 n + 2r_2 - r_1) / 2}. 
\end{align} 
Particular special cases of this series 
include the expansions corresponding to the tuples 
$T \equiv (a,b,r_2,r_1,r_0)$ in the series forms of 
\eqref{eqn_ParamThetaFn_Trsabcq_BilateralSeries_def_stmt-intro_v1} and 
\eqref{eqn_ParamThetaFn_Trsabcq_BilateralSeries_def_stmt-intro_v3} 
for the following functions: the series for the 
\emph{Dedekind eta function}, $\eta(\tau)$, 
expanded in the bilateral series form of 
$T = (0,1,3,1,0)$ 
\citep[\S 23.17]{NISTHB}, 
a pair of bilateral square series related to the 
\emph{Rogers--Ramanujan continued fraction}, $R(q)$, expanded by the 
series where $T = (10,3,5,3,0)$ and $T = (10,1,5,1,0)$ 
\citep[\cf \Section 5]{INTRORAMLOSTNB} 
\citep[\cf \Section 16]{BERNDTRAMNBIII}, the 
parameters $T = (0, 1, 25, -15, 0)$ for the coefficients of the 
$q$--series $(q)_{\infty}^3$ modulo $5$, and the 
cases where $T = (0, 1, 3, \pm 1, 0)$ related to the 
\textit{Euler function}, $\phi(q) \equiv (q)_{\infty}$, and the 
\textit{pentagonal number theorem} 
\citep[Thm. 353]{HARDYWRIGHTNUMT} \citep[Cor. 1.3.5]{NUMTSPRAM} 
\citep[\cf Entry 22(iii)]{BERNDTRAMNBIII}. 
Even further examples of the series expansions phrased in terms of 
\eqref{eqn_ParamThetaFn_Trsabcq_BilateralSeries_def_stmt-intro_v1} 
are compared in the references 
\citep[\cf Cors. 1.3.21 and 1.3.22]{NUMTSPRAM} 
\citep[\cf Thms. 355 and 356; (19.9.3)]{HARDYWRIGHTNUMT} 
\citep[\cf \Section 17.2(vi); (17.8.3)]{NISTHB} 
\citep{BERNDTRAMNBIII,RRIDENTLIST}. 

\section{Statement and Constructions of the Main Theorem} 
\label{subSection_Intro_SeriesTransforms_Involving_S2Nums} 
\label{subsubSection_Intro_Sketch_SqSeriesExps_Construction} 
\label{subsubSection_Intro_Stmts_OfThe_MainThmResults} 
\label{Section_SquareSeriesGF_transforms} 
\label{Section_SquareSeriesGF_transforms-v0} 

The \textit{Stirling numbers of the second kind}\footnote{ 
     The triangle is also commonly denoted by $S(n, k)$ as in 
     \citep[\S 26.8]{NISTHB}. The 
     bracket notation, $\gkpSII{n}{k}$, for these coefficients is 
     consistent with the 
     notation employed by the reference \citep[\cf \S 6]{GKP}. 
}, 
$\gkpSII{n}{k}$, are related to a specific key transformation result of the 
OGF, $F(z)$, of an arbitrary sequence stated in the next proposition. 
In particular, the triangular recurrence relation defining these numbers 
gives rise to the finite expression in 
\eqref{eqn_nkpow_Fz_S2OGF_exp_ident} below 
for the generating function of the 
modified sequence, $\langle n^k f_n \rangle$, when $k \in \mathbb{Z}^{+}$ 
given only in terms of a sum over the Stirling numbers and the higher-order 
$j^{th}$ derivatives of the original sequence generating function when these 
derivatives exist for all $j \leq k$. 
The next result is an important ingredient to the proof of 
Theorem \ref{thm_SqSeries_OGF_Transforms} 
given in this section. 

\begin{prop}[Generating Function Transformations Involving Stirling Numbers]
\label{prop_S2_OGFFz_to_nPowmOGFFz_stmt} 
For any fixed $k \in \mathbb{Z}^{+}$, let the ordinary generating function, 
$F_f(z)$, of an arbitrary sequence, $\langle f_n \rangle$, be defined 
such that its $j^{th}$-order derivatives, $F^{(j)}(z)$, 
exist for all $j \leq k$. Then for this fixed $k \in \mathbb{Z}^{+}$, 
we have a finite expression for the OGF of the modified sequence, 
$\langle n^m f_n \rangle$, given by the following transformation identity: 
\begin{equation} 
\label{eqn_nkpow_Fz_S2OGF_exp_ident} 
\sum_{n=0}^{\infty} n^k f_n z^n = \sum_{j=0}^{k} \gkpSII{k}{j} z^j F^{(j)}(z). 
\end{equation} 
\end{prop} 
\begin{proof} 
The proof follows by induction on $k \geq 0$. 
For the base case of $k = 0$, the 
right--hand--side of \eqref{eqn_nkpow_Fz_S2OGF_exp_ident} evaluates to the 
original OGF, $F(z)$, as claimed, since $\gkpSII{0}{0} = 1$. 
The Stirling numbers of the second kind are defined in 
\citep[\S 6.1]{GKP} by the triangular recurrence relation in 
\eqref{eqn_S2_rdef} for integers $n, k \geq 0$. 
\begin{equation} 
\label{eqn_S2_rdef} 
\gkpSII{n}{k} = k \gkpSII{n-1}{k} + \gkpSII{n-1}{k-1} + \Iverson{n = k = 0} 
\end{equation} 
Suppose that 
\eqref{eqn_nkpow_Fz_S2OGF_exp_ident} holds for some $k \in \mathbb{N}$. 
For this fixed $k$, let $F(z)$ denote the OGF of the sequence 
$\langle f_n \rangle$ such that the $j^{th}$ derivatives of the 
function exist for all $j \in \{0, 1, 2, \ldots, k+1\}$. 
%Since \eqref{eqn_nkpow_Fz_S2OGF_exp_ident} is true for $k$ by assumption, 
%this result can be applied to the 
%modified sequence of terms $\langle n f_n \rangle$ with the 
%corresponding OGF of $z F^{\prime}(z)$. 
It follows from our assumption that 
\begin{align} 
\notag 
\sum_{n=0}^{\infty} n \cdot \left(n^k f_n\right) z^n & = z \frac{d}{dz}\Biggl[ 
     \sum_{j=0}^{k} \gkpSII{k}{j} z^j F^{(j)}(z)\Biggr] \\ 
\notag 
    & = 
\sum_{j=0}^{k} \gkpSII{k}{j} \left(j z^j F^{(j)}(z) + z^{j+1} F^{(j+1)}(z) 
     \right) \\ 
\notag 
   & = 
     \sum_{j=0}^{k} j \gkpSII{k}{j} z^j F^{(j)}(z) + 
     \sum_{j=1}^{k+1} \gkpSII{k}{j-1} z^j F^{(j)}(z) \\ 
\label{eqn_S2TransformProof_nPowk_LHS_sum} 
   & = 
     \sum_{j=1}^{k+1} \left(j \gkpSII{k}{j} + \gkpSII{k}{j-1}\right) 
     z^j F^{(j)}(z) 
\end{align} 
where $\gkpSII{k}{0} \equiv 0$ for all non-negative $k \neq 0$. 
Finally, we equate the left-hand-side of 
\eqref{eqn_S2_rdef} to the inner right-hand-side terms in 
\eqref{eqn_S2TransformProof_nPowk_LHS_sum} to obtain that 
\begin{align} 
\label{eqn_S2TransformProof_nPowkp1_sum} 
\sum_{j=0}^{k+1} \gkpSII{k+1}{j} z^j F^{(j)}(z) & = 
     \sum_{n=0}^{\infty} n^{k+1} f_n z^n. 
\end{align} 
Thus the sum in \eqref{eqn_nkpow_Fz_S2OGF_exp_ident} expanded in terms of the 
Stirling number triangle holds for all $k \in \mathbb{N}$. 
\end{proof} 

\begin{remark}[Applications of the Stirling Number Transformations]
\label{remark_GenFormsOfS2GFTransforms}
The series transformations given by the result in 
\eqref{eqn_nkpow_Fz_S2OGF_exp_ident} 
effectively generalize the following known identity for the negative-order 
polylogarithm function corresponding to the 
special case sequences of $f_n \equiv c^n$ that holds whenever $|cz| < 1$ 
\citep[Table 351; \S 7.4]{GKP}: 
\begin{equation} 
\notag 
\Li_{-m}(cz) \equiv 
     \sum_{n=0}^{\infty} n^{m} (cz)^{n} = 
     \sum_{j=0}^{m} \gkpSII{m}{j} \frac{(cz)^{j} \cdot j!}{(1-cz)^{j+1}},\ 
     m \in \mathbb{Z}^{+}. 
\end{equation} 
The Stirling number identity in the proposition 
is of particular utility in 
finding generating functions or other closed--form expressions for the series 
$\sum_{n} g(n) f_n z^n$ where the function $g(n)$ is expanded as an 
infinite power series in the variable $n$. 
For example, for any $q \in \mathbb{C}$ and any natural number $n \geq 0$ 
we may expand the square-power terms, $g(n) = q^{n^2}$, as 
\begin{equation} 
\label{eqn_qPowNSquared_ExpFn_TaylorSeriesExp_in_nSquared} 
q^{n^2} \equiv \exp\left(n^2 \Log(q)\right) = 
     \sum_{i=0}^{\infty} \frac{\Log(q)^{i} n^{2i}}{i!}. 
\end{equation} 
This particular expansion is key to the proof of 
Theorem \ref{thm_SqSeries_OGF_Transforms} immediately below. 
Other examples where this approach applies include the 
lacunary sequence case where $g(n) = q^{2^n}$ for some $|q| < 1$. 
\end{remark} 

\begin{theorem}[Square Series Transformations] 
\label{thm_SqSeries_OGF_Transforms} 
Fix $q \in \mathbb{C}$ and 
suppose that $\langle f_n \rangle$ denotes a prescribed sequence whose 
ordinary generating function, $F_f(z)$, is 
analytic for $z$ on a corresponding non-trivial region of 
$R_{\sq}\left(f, q\right)$. 
Then the unilateral square series functions, $F_{\sq}(f; q, z)$, defined by 
\eqref{eqn_Unilateral_Fsq_SqSeries_GF} satisfy the two formulas 
\begin{subequations} 
\label{eqn_MainThmStmt_OrdExpSqSeriesIntReps-stmts_v1} 
\begin{align} 
\label{eqn_MainThmStmt_OrdSqSeriesIntReps_v1} 
F_{\sq}(f; q, z) & = 
     \int_0^{\infty} \frac{e^{-t^2/2}}{\sqrt{2\pi}} \left[ 
     \sum_{b=\pm 1} \sum_{j=0}^{\infty} 
     \frac{z^j \cdot F_f^{(j)}(z)}{j!} \times 
     \left(e^{bt \sqrt{2 \Log(q)}} - 1\right)^{j} 
     \right] dt \\ 
\label{eqn_MainThmStmt_OrdSqSeriesIntReps_v2} 
   & = 
     \int_0^{\infty} \frac{e^{-t^2/2}}{\sqrt{2\pi}} \left[ 
     \sum_{b=\pm 1} F_f\left( 
     e^{bt \sqrt{2 \Log(q)}} \cdot z 
     \right) 
     \right] dt, 
\end{align} 
\end{subequations} 
for any $q \in \mathbb{C}$ such that $|q| \leq 1$ and 
$z \in R_{\sq}\left(f, q\right)$ such that the square series expansion 
has a non-trivial radius of convergence. 
\end{theorem} 
\begin{proof} 
We begin by expanding out the series in 
\eqref{eqn_Unilateral_Fsq_SqSeries_GF} 
in terms of the $j^{th}$ derivatives of the OGF, $F_f(z)$, for the 
sequence, $\langle f_n \rangle$, and an infinite series over the 
Stirling numbers of the second kind as 
\begin{align} 
\notag 
\sum_{n=0}^{\infty} q^{n^2} f_n\ z^n & = 
\sum_{n=0}^{\infty} \left(\sum_{k=0}^{\infty} 
     \frac{\Log(q)^k n^{2k}}{k!}\right) f_n\ z^n \\ 
\notag 
     & = 
     \sum_{k=0}^{\infty} \frac{\Log(q)^k}{k!} \left( 
     \sum_{n=0}^{\infty} n^{2k} f_n z^n\right) \\ 
\notag 
     & = 
     \sum_{k=0}^{\infty} \left(\frac{\Log(q)^k}{k!} 
     \sum_{j=0}^{2k} \gkpSII{2k}{j} F^{(j)}(z) z^j\right) \\ 
   & = 
\label{eqn_cnsquared_S2_j_sum} 
     \sum_{j=0}^{\infty} \left(\sum_{k=0}^{\infty} \gkpSII{2k}{j} 
     \frac{\Log(q)^k}{k!}\right) z^j F^{(j)}(z). 
\end{align} 
where the terms $\gkpSII{2k}{j}$ are necessarily zero for all $j > 2k$. 

We next proceed to prove an integral transformation for the inner 
Stirling number series in \eqref{eqn_cnsquared_S2_j_sum}. 
In particular, we claim that 
\begin{align} 
\label{eqn_lemma_formulas_for_reciprocals_of_nFact_in_2nFact-stmts_v1} 
\frac{\Log(q)^{i}}{i!} & = \frac{1}{(2i)!} \times \left[ 
     \int_0^{\infty} \frac{2 e^{-t^2/2}}{\sqrt{2\pi}} 
     \left(\sqrt{2 \Log(q)} \cdot t \right)^{2i} dt 
     \right], 
\end{align} 
and then that 
\begin{align} 
\label{eqn_proof_of_thm_claim2} 
\sum_{k \geq 0} \gkpSII{2k}{j} \frac{\Log(q)^k}{k!} & = 
     \frac{1}{\sqrt{2\pi}} \int_0^\infty \left[ 
     \sum_{b = \pm 1} \frac{1}{j!} \left(e^{\sqrt{2 \Log(q)} t}-1\right)^j 
     \right] e^{-t^2 / 2} dt. 
\end{align} 
To prove these claims, first notice that the next identity follows from 
the \emphonce{duplication formula} for the gamma function as 
\citep[\S 5.5(iii)]{NISTHB} 
\begin{equation} 
\label{eqn_2nSingleFact_NFact2PowGammaFn_exp_ident} 
(2n)! = \frac{4^n n!}{\sqrt{\pi}} \Gamma\left(n+\frac{1}{2}\right), 
\end{equation} 
where an integral representation for the gamma function, $\Gamma(n + 1/2)$, is 
given by \citep[\S 5.9(i)]{NISTHB} 
\begin{equation} 
\label{eqn_GammaFn_nPlus1Over2_integral_rep} 
\Gamma\left(n+\frac{1}{2}\right) = 2^{n-1/2} \int_0^{\infty} 
     t^{2n} e^{-t^2 / 2}\ dt. 
\end{equation} 
A well-known exponential generating function for the Stirling numbers of the 
second kind, $\gkpSII{2k}{j}$, at any fixed natural number 
$j \geq 0$ is given by \citep[\S 7.2, \S 7.4]{GKP} 
\begin{align*} 
\sum_{k \geq 0} \gkpSII{2k}{j} \frac{z^{2k}}{(2k)!} & = 
     \frac{1}{2 j!}\left[(e^z - 1)^j + (e^{-z} - 1)^j\right]. 
\end{align*} 
Then we have from \eqref{eqn_cnsquared_S2_j_sum} and 
\eqref{eqn_proof_of_thm_claim2} that 
\begin{align*} 
\sum_{n=0}^{\infty} q^{n^2} f_n\ z^n & = 
     \sum_{j=0}^{\infty} \left( 
     \frac{1}{\sqrt{2\pi}} \int_0^\infty \left[ 
     \sum_{b = \pm 1} \frac{1}{j!} \left(e^{\sqrt{2 \Log(q)} t}-1\right)^j 
     \right] e^{-t^2 / 2} dt. 
     \right) \times z^j F^{(j)}(z). 
\end{align*} 
It is not difficult to show by the Weierstrass M-test that for 
bounded $\Log(q) \in \mathbb{C}$, the sum over the $j^{th}$ powers in 
the previous equation with respect to $j \geq 0$ converges uniformly. 
This implies that we can interchange the order of 
summation and integration to obtain the first result stated in 
\eqref{eqn_MainThmStmt_OrdSqSeriesIntReps_v1}. 
The second equivalent statements of the theorem in 
\eqref{eqn_MainThmStmt_OrdSqSeriesIntReps_v2} 
is an immediate consequence of the first result in 
\eqref{eqn_MainThmStmt_OrdSqSeriesIntReps_v1}. 
\end{proof} 

\begin{table}[ht] 

\begin{center} 

\setlength{\fboxrule}{1.25pt} 
\setlength{\fboxsep}{4pt} 
\fbox{\begin{minipage}{\textwidth} 
     
     \newcommand{\SqSeriesFnTableSubHeader}[1]{
          \taburowcolors 2{Plum!32!White .. Plum!10!White} 
          \hline 
          \multicolumn5{|l|}{\bfseries\normalsize{#1}} \\ \hline\hline 
          \rowfont{\bfseries\small\mathversion{bold}} Eq. & 
               $f_n$ & $F_f(z)$ & 
               $F_f^{(j)}(z)$ & Parameters \\ \hline%\hline 
          \taburowcolors 2{Plum!0!White .. Plum!0!White} 
     } 
     \newcommand{\AddTableHeader}[1]{\SqSeriesFnTableSubHeader{#1}} 

     \tabulinestyle{2.5pt OTblBC} 
     \taburulecolor |OTblBC|{ITblBC} 
     \arrayrulewidth=1.5pt \doublerulesep=0pt \tabulinesep=3pt%\extrarowsep=3pt
     \begin{tabu} to \TblW {|c |T{+1}{c}|T{+1}{c}|T{+1}{c}|X[l]|} \hline 

       \AddTableHeader{Geometric--Series--Based Sequences} 

       \LabelTableEqn{eqn_Table_SpCaseSequenceOGFFormulas-GeomSeriesBased_v1} & 
           $c^n$ & 
           $\frac{1}{(1-cz)}$ & 
           $\frac{c^j j!}{(1-cz)^{j+1}}$ & 
           $|cz| < 1$ \\ 
       \hline 

       \LabelTableEqn{eqn_Table_SpCaseSequenceOGFFormulas-GeomSeriesBased_v2} & 
           $n c^n$ & 
           $\frac{cz}{(1-cz)^{2}}$ & 
           $\frac{c^j j! \left(cz + j\right)}{(1-cz)^{j+2}}$ & 
           $|cz| < 1$ \\ 
       \hline 

       \LabelTableEqn{eqn_Table_SpCaseSequenceOGFFormulas-GeomSeriesBased_v5} & 
           $n^{2} c^n$ & 
           $\frac{cz(cz+1)}{(1-cz)^{3}}$ & 
           $\frac{c^j j! \left((cz)^2 + (3j+1) cz + j^2\right)}{(1-cz)^{j+3}}$ & 
           $|cz| < 1$ \\ 
       \hline 

       \LabelTableEqn{eqn_Table_SpCaseSequenceOGFFormulas-GeomSeriesBased_v3} & 
           $(an+b) c^n$ & 
           $\frac{(a-b) cz+b}{(1-cz)^{2}}$ & 
           $\frac{c^{j} j! \times \left((a-b) cz + aj+b\right)}{ 
            (1-cz)^{j+2}}$ & 
           $|cz| < 1$; $a,b \in \mathbb{C}$ \\ 
       \hline 

       \LabelTableEqn{eqn_Table_SpCaseSequenceOGFFormulas-GeomSeriesBased_v4} & 
           $n(n-1) c^n$ & 
           $\frac{2 (cz)^2}{(1-cz)^{3}}$ & 
           $\frac{c^j j!\left(2 (cz)^2 + 4j cz + j(j-1)\right)}{(1-cz)^{j+3}}$ & 
           $|cz| < 1$ \\ 
       \hline 

       \AddTableHeader{Exponential--Series--Based Sequences} 

       \LabelTableEqn{eqn_Table_SpCaseSequenceOGFFormulas-ExpSeriesBased_v1} & 
           $\frac{r^n}{n!}$ & 
           $e^{rz}$ & 
           $r^{j} e^{rz}$ & 
           $r \in \mathbb{C}$ \\ 
       \hline 

       \LabelTableEqn{eqn_Table_SpCaseSequenceOGFFormulas-ExpSeriesBased_v2} & 
           $\frac{n(n-1) r^n}{n!}$ & 
           $(rz)^2 e^{rz}$ & 
           $\mathsmaller{r^j e^{rz} \left[(rz)^2 + 2j rz + j(j-1)\right]}$ & 
           $r \in \mathbb{C}$ \\ 
       \hline\hline 

     \end{tabu} 

     \setcaption{table}{Special Case Sequences and Generating Function Formulas}{ 
                   Listings of the ordinary generating functions, $F_f(z)$, 
                   $j^{th}$ derivative formulas, 
                   $F_f^{(j)}(z)$ for $j \in \mathbb{N}$, and 
                   corresponding series parameters 
                   for several special case sequences, $\langle f_n \rangle$.} 
     \label{table_fnSeq_OGFs_and_jthDerivs_listings-stmts_v1} 

\end{minipage}} 
\end{center} 

\end{table} 

\begin{remark}[Admissible Forms of Sequence Generating Functions]
In practice, the 
assumption made in phrasing 
Proposition \ref{prop_S2_OGFFz_to_nPowmOGFFz_stmt}
and in the 
theorem statement requiring that the sequence OGF is analytic on some 
non-trivial $R_0(f)$, \ie so that $F_f(z) \in C^{\infty}(R_0(f))$, 
does not impose any significant or particularly restrictive conditions on the 
applications that may be derived from these transformations. 
Rather, provided that suitable choices of the 
series parameters in the starting OGF series 
for many variations of the geometric-series-based and 
exponential-series-based sequence forms lead to convergent generating 
functions analytic on some open disk, the 
formulas for all non-negative integer-order OGF derivatives exist, and 
moreover, are each easily obtained in closed-form as 
functions over all $j \geq 0$. 

Table \ref{table_fnSeq_OGFs_and_jthDerivs_listings-stmts_v1} 
lists a few such closed-form formulas 
for the $j^{th}$ derivatives of such sequence OGFs 
corresponding to a few particular classes of these square series expansions 
that arise frequently in our applications. 
The particular variations of the geometric square series in the forms of the 
generalized functions defined in 
Table \ref{table_Notation_SeriesDefsOfSpecialSqSeries-stmts_v1} 
on page \pageref{table_Notation_SeriesDefsOfSpecialSqSeries-stmts_v1} 
arise naturally in applications including the study of theta functions, 
modular forms, in infinite products that form the 
generating functions for partition functions, and in other 
combinatorial sequences of interest within the article. 
The new integral representations for the generalized classes of 
series expansions for the functions defined by 
Table \ref{table_Notation_SeriesDefsOfSpecialSqSeries-stmts_v1} 
given in the next sections are derived from 
Theorem \ref{thm_SqSeries_OGF_Transforms} and from the 
formulas for the particular special case sequence OGFs given separately in 
Table \ref{table_fnSeq_OGFs_and_jthDerivs_listings-stmts_v1}. 
The next sections provide applications extending the examples 
already given in Section \ref{subSection_Intro_Examples}. 
\end{remark} 

\section{Applications of the Geometric Square Series} 
\label{Section_Applications_of_GeomSqSeries} 

\subsection{Initial Results} 

In this subsection, we state and prove initial results providing 
integral representations for a few generalized classes of the 
geometric square series cases where $f_n := c^n$ which follow from 
Theorem \ref{thm_SqSeries_OGF_Transforms}. 
Once we have these generalized propositions at our disposal, we move along 
to a number of more specific applications of these results in the 
next subsections which provide new 
integral representations for the Jacobi theta functions and other 
closely-related special function series and constants. 
The last subsection of this section proves corresponding integral 
representations for the polynomial multiples of the geometric 
square series where $f_n := (\alpha n+\beta)^m \cdot c^n$ which have 
immediate applications to the expansions of higher-order derivatives of the 
Jacobi theta functions. 

\StartGroupingSubEquations{} 
\label{eqn_geom_series_initial_integrals} 

\begin{prop}[Geometric Square Series] 
\label{prop_OrdExp_GeomSquareSeries} 
Let $q \in \mathbb{C}$ be defined such that $|q| < 1$ and 
suppose that $c,z \in \mathbb{C}$ 
are defined such that $|cz| < 1$. 
For these choices of the series parameters, $q,c,z$, the 
\keywordemph{ordinary geometric square series}, $G_{\sq}(q, c, z)$, 
defined by \eqref{eqn_Ordinary_and_ExpGeomSqSeries_defs_stmts_v2} of 
Table \ref{table_Notation_SeriesDefsOfSpecialSqSeries-stmts_v1} 
satisfies an integral representation of the following form: 
%% See : "temp-working-2013.03.26-v1.*": 
\begin{align} 
\label{eqn_geom_series_initial_integrals-form_v2} 
G_{\sq}(q, c, z) & = 
     \int_0^{\infty} \frac{2 e^{-t^2/2}}{\sqrt{2\pi}} \left[ 
     \frac{1 - cz \cosh\left(t \sqrt{2 \Log(q)}\right)}{ 
     c^2 z^2 - 2cz \cosh\left(t \sqrt{2 \Log(q)}\right) + 1} 
     \right] dt. 
\end{align} 
\end{prop} 
\begin{proof}[Proof Sketch] 
The general method of proof employed within this section is given 
along the lines of the second phrasing of the transformation results in 
\eqref{eqn_MainThmStmt_OrdSqSeriesIntReps_v2} of 
Theorem \ref{thm_SqSeries_OGF_Transforms}. 
Alternately, we are able to arrive in these results as in the 
supplementary \emph{Mathematica} notebook \citep{SQS-SUMMARYNB-REF}
by applying the 
formula involving the sequence OGFs in 
\eqref{eqn_MainThmStmt_OrdSqSeriesIntReps_v1} 
using the results given in 
Table \ref{table_fnSeq_OGFs_and_jthDerivs_listings-stmts_v1} 
where the series parameter $z$ is taken over the 
ranges stated in the rightmost columns of the table. 
\renewcommand{\qedsymbol}{$\boxdot$}
Thus we may use both forms of the transformation result stated in 
the theorem interchangeably 
to arrive at proofs of our new results in subsequent sections of the article. 
\end{proof} 
\begin{proof} 
We prove the particular result for this formula in 
\eqref{eqn_geom_series_initial_integrals-form_v2} 
as a particular example case of the 
first similar argument method to be employed in the further results given 
in this section below. 
In particular, for $|cz| < 1$, the 
main theorem and the OGF for the special case of the 
geometric-series-based sequence, $f_n \equiv c^n$, given in 
\eqref{eqn_Table_SpCaseSequenceOGFFormulas-GeomSeriesBased_v1} 
of the table imply that 
\begin{align} 
\label{eqn_geom_series_initial_integrals-form_v1} 
G_{\sq}(q, c, z) & = 
     \int_0^{\infty} \frac{e^{-t^2/2}}{\sqrt{2\pi}} \left[ 
     \sum_{b=\pm 1} \left(1-e^{bt \sqrt{2 \Log(q)}} cz\right)^{-1} 
     \right] dt. 
\end{align} 
The separate integrand functions in the previous equation are combined as 
\begin{align} 
\tag{i} 
\label{eqn_ProofOf_GeomSqSeriesFirstIntRepProp-tagged_formula_v.i} 
\sum_{b=\pm 1} \left(1-e^{bt \sqrt{2 \Log(q)}} cz\right)^{-1} & = 
     \frac{ 
     2 - cz \left[e^{t \sqrt{2 \Log(q)}} + e^{-t \sqrt{2 \Log(q)}}\right]}{ 
     c^2z^2 - \left[e^{t \sqrt{2 \Log(q)}} + e^{-t \sqrt{2 \Log(q)}}\right] 
     + 1}, 
\end{align} 
Since the \keywordemph{hyperbolic cosine function} satisfies 
\citep[\S 4.28]{NISTHB} 
\begin{align} 
\label{eqn_ProofOf_GeomSqSeriesFirstIntRepProp-cosh_ident_formula-tag_v.ii} 
2 \times \cosh\left(t \sqrt{2 \Log(q)}\right) & = 
     2 \times \left[ 
     \sum_{b=\pm 1} \frac{e^{bt \sqrt{2 \Log(q)}}}{2} 
     \right], 
\end{align} 
the integral formula claimed in the proposition then follows by 
rewriting the intermediate exponential terms in the combined formula from 
\eqref{eqn_ProofOf_GeomSqSeriesFirstIntRepProp-tagged_formula_v.i} 
given above. 
\end{proof} 

\begin{remark}[Integrals for Shifted Forms of the Geometric Square Series] 
\label{remark_} 
\label{prop_GenGeomSqSeries_integral_stmt} 
For any $d \in \mathbb{Z}^{+}$, we note that the 
slightly modified functions, $\vartheta_{d}(q, c, z)$, defined as the 
\keywordemph{``shifted'' geometric square series} from 
\eqref{eqn_Theta_cqz_GeomSeries-Based_series_stmt} in 
Table \ref{table_Notation_SeriesDefsOfSpecialSqSeries-stmts_v1} 
satisfy expansions through $G_{\sq}(q, c, z)$ of the form 
\begin{align} 
\notag 
\widetilde{\vartheta}_{d}(q, c, z) & = 
     \left(q^{d} cz\right)^{d} \times G_{\sq}\left(q, q^{2d} c, z\right). 
\end{align} 
Moreover, for any fixed choice of $d \in \mathbb{Z}^{+}$, 
we may similarly obtain modified forms of the integral representations in 
Proposition \ref{prop_OrdExp_GeomSquareSeries} 
as follows: 
\begin{align} 
\notag 
G_{\sq}(q, c, z) & = 
     \sum_{i=0}^{d-1} q^{i^2} (cz)^{i} + \widetilde{\vartheta}_{d}(q, c, z), 
     \text{ if } 
     |cz| < \min\left(1, |q|^{-2d}\right). 
\end{align} 
If we then specify that 
$|cz| < \min\left(1, |q|^{-2d}\right)$ 
for any fixed integer $d \geq 0$, 
we may expand integral representations for the shifted series variants 
of these functions through the formulas given in 
Proposition \ref{prop_OrdExp_GeomSquareSeries} as 
%% : "c_nsquared_integral-2011.05.18.*": 
%% : "c_nsquared_integral-2013.05.12-v1.*": 
\begin{align} 
\label{eqn_geom_series_Theta_cqz_integral_stmt-Tmcqz_v1} 
\widetilde{\vartheta}_d(c, q, z) & = 
     \left(q^{d} cz\right)^{d} 
     \int_0^{\infty} \frac{e^{-t^2/2}}{\sqrt{2\pi}} \left[\frac{ 
     2 -  2 q^{2d} cz \cosh\left(t \sqrt{2 \Log(q)}\right)}{ 
     q^{4d} c^2 z^2 - 2 q^{2d} cz \cosh\left(t \sqrt{2 \Log(q)}\right) + 1} 
     \right] dt. 
\end{align} 
The series for the special case of the shifted series when $d := 1$ 
arises in the unilateral series expansions of 
bilateral square series of the form 
\begin{align} 
\notag 
\sum_{n=-\infty}^{\infty} \left(\pm 1\right)^{n} q^{n(r_1 n+r_0)/2} & = 
     1 + \sum_{n=1}^{\infty} \left(\pm 1\right)^{n} q^{n(r_1 n+r_0)/2} + 
      \sum_{n=1}^{\infty} \left(\pm 1\right)^{n} q^{n(r_1 n-r_0)/2}. 
\end{align} 
Whenever $|cz| < \min\left(1, |q|^{-2}\right)$, 
this shifted square series function satisfies the special case 
integral formula given by 
\begin{align} 
\label{eqn_geom_series_Theta_cqz_integral_stmt-SpCase_dEq1-Tmcqz_v2} 
\widetilde{\vartheta}_1(c, q, z) & = 
     \int_0^{\infty} \frac{e^{-t^2/2}}{\sqrt{2\pi}} \left[\frac{ 
     2qcz -  2 q^{3} c^2 z^2 \cosh\left(t \sqrt{2 \Log(q)}\right)}{ 
     q^{4} c^2 z^2 - 2 q^{2} cz \cosh\left(t \sqrt{2 \Log(q)}\right) + 1} 
     \right] dt. 
\end{align} 
\end{remark} 

\EndGroupingSubEquations{} 

\begin{cor}[A Special Case] 
\label{prop_Qsq_abcqz_unilateral_series_fn_integral_rep_v1} 
Suppose that $a, b \in \mathbb{C}$ are fixed scalars and that 
$c, z \in \mathbb{C}$ are chosen such that $|cz| < 1$. 
For these choices of the series parameters in the definition of 
\eqref{eqn_GeomSqSeries_Qabqcz-intro_series_exp_v1}, the 
square series function, $Q_{a,b}(q, c, z)$, 
has an integral representation of the form 
\begin{align} 
\label{eqn_Qsq_abcqz_unilateral_series_fn_integral_rep_v1} 
Q_{a,b}(q, c, z) & = \int_0^{\infty} 
     \frac{2acz e^{-t^2/2}}{\sqrt{2\pi}} \left[ 
     \frac{\left(c^2 z^2 + 1\right) \cosh\left(t \sqrt{2 \Log(q)}\right) - 2cz}{ 
     \left(c^2 z^2 - 2cz \cosh\left(t \sqrt{2 \Log(q)}\right) + 1\right)^2} 
     \right] dt \\ 
\notag 
   & \phantom{= \int_0^{\infty} } + 
     b \cdot G_{\sq}(q, c, z), 
\end{align} 
where the second term 
in \eqref{eqn_Qsq_abcqz_unilateral_series_fn_integral_rep_v1} 
given in terms of the function, $G_{\sq}(q, c, z)$, 
satisfies the integral representation from 
Proposition \ref{prop_OrdExp_GeomSquareSeries} given above. 
\end{cor} 
\begin{proof} 
The result follows by applying the theorem 
to the sequence OGF formulas in 
\eqref{eqn_Table_SpCaseSequenceOGFFormulas-GeomSeriesBased_v1} and 
\eqref{eqn_Table_SpCaseSequenceOGFFormulas-GeomSeriesBased_v2}, 
respectively, and then combining the denominators in the second integrand using 
\eqref{eqn_ProofOf_GeomSqSeriesFirstIntRepProp-cosh_ident_formula-tag_v.ii} 
as in the first proof of 
Proposition \ref{prop_OrdExp_GeomSquareSeries}. 
Equivalently, we notice that the integrand in 
\eqref{eqn_geom_series_initial_integrals-form_v2} 
of the previous proposition may be differentiated termwise with 
respect to $z$ to form the first term in the new integral representation 
formulated by \eqref{eqn_Qsq_abcqz_unilateral_series_fn_integral_rep_v1}. 
In this case, the 
function $Q_{a,b}(q, c, z)$ defined as the power series in 
\eqref{eqn_GeomSqSeries_Qabqcz-intro_series_exp_v1} 
is related to the function, $G_{\sq}(q, c, z)$, by 
\begin{equation} 
\notag 
Q_{a,b}(q, c, z) = az \cdot G_{\sq}^{\prime}(q, c, z) + 
     b \cdot G_{\sq}(q, c, z), 
\end{equation} 
where the right-hand-side derivative is taken with respect to $z$, 
and where $G_{\sq}(q, c, z)$ may be differentiated 
termwise with respect to $z$. 
In particular, 
the first derivative of the function, $G_{\sq}(q, c, z)$, is 
obtained through 
Proposition \ref{prop_OrdExp_GeomSquareSeries} by differentiation as 
\begin{align} 
\notag 
G_{\sq}^{\prime}(q, c, z) & = 
     \int_0^{\infty} 
     \frac{2 e^{-t^2/2}}{\sqrt{2\pi}} \times %\left[ 
     \frac{\partial}{\partial{z}}\left\lbrace 
     \frac{1 - cz \cosh\left(\sqrt{2 \Log(q)} t\right)}{ 
     c^2 z^2 - 2cz \cosh\left(\smash[bt]{\sqrt{2 \Log(q)}} t\right) + 1} 
     \right\rbrace 
     %\right] 
     dt \\ 
\notag 
   & = 
     \int_0^{\infty} 
     \frac{2c e^{-t^2/2}}{\sqrt{2\pi}} \left[ 
     \frac{\left(c^2 z^2 + 1\right) \cosh\left(\sqrt{2 \Log(q)} t\right) - 2cz}{ 
     \left(c^2 z^2 - 2cz \cosh\left(\sqrt{2 \Log(q)} t\right) + 1\right)^2} 
     \right] dt. 
     \qedhere 
\end{align} 
\end{proof} 

\subsection{Jacobi Theta Functions and Related Special Function Series} 
\label{subsubSection_Apps_SpCases_of_the_JacobiThetaFns} 
\label{subSection_Appls_Ints_for_SPFns_Related_to_the_GeomSqSeries} 

The \emphonce{geometric square series}, $G_{\sq}(q, c, z)$, 
is related to the form of many special functions where the parameter 
$z := p(q)$ is some fixed rational power of $q$. 
A number of 
such expansions of these series correspond to special cases of this 
generalized function, denoted by $G_{\sq}\left(p, m; q, \pm c\right)$, 
expanded through its integral representation below 
where the fixed $p$ and $m$ 
are considered to assume some values chosen over the 
positive rationals. 
These integral representations, of course, do not lead to simple 
power series expansions in $q$ of the integrand about zero, and 
cannot be integrated 
termwise with respect to $q$ in their immediate form following from 
Proposition \ref{prop_OrdExp_GeomSquareSeries}. 

Notice that 
whenever $p, m, q, c \in \mathbb{C}$ are taken 
such that $|q^{p}| \in (0, 1)$ and $|q^{m} c| < 1$, 
an integral representation for the function, 
$G_{\sq}\left(p, m; q, \pm c\right) \equiv 
 Gi_{\sq}\left(q^p, c \cdot q^m, \pm 1\right)$, 
defined by the series from 
\eqref{eqn_GeomSqSeries_SpFnVariants_GsqpmqcFn_series_def_v1} follows from 
Proposition \ref{prop_OrdExp_GeomSquareSeries} as
\begin{align} 
\label{eqn_GeomSqSeries_SpFnVariants_GsqpmqcFn_ModIntRep_v2} 
G_{\sq}\left(p, m; q, \pm c\right) & = 
     \int_0^{\infty} \frac{2 e^{-t^2/2}}{\sqrt{2\pi}} \left[ 
     \frac{\left(1 \mp c q^m \cosh \left(\sqrt{2p \Log(q)} t\right) 
     \right)}{c^2 q^{2m} \mp 2 c q^{m} z \cosh\left( 
     \sqrt{2p \Log(q)} t\right) + 1} \right] dt. 
\end{align} 
The series for the Jacobi theta functions, 
$\vartheta_i(q) \equiv \vartheta_i(0, q)$ when $i = 2,3,4$, and 
\emph{Ramanujan's functions}, $\varphi(q)$ and $\psi(q)$, 
are examples of these special case series 
that are each cited as particular applications of the result in 
\eqref{eqn_GeomSqSeries_SpFnVariants_GsqpmqcFn_ModIntRep_v2} 
in the next subsections of the article. 

\subsubsection*{Integral Representations for the Jacobi Theta Functions, $\vartheta_i(q)$} 

The variants of the classical Jacobi theta functions, 
$\vartheta_i(u, q)$ for $i = 1,2,3,4$, defined by the functions, 
$\vartheta_i(q) \equiv \vartheta_i(0, q)$ where 
$\vartheta_i(q) \nequiv 0$ for $i = 2,3,4$ 
are expanded through the asymmetric, unilateral Fourier series for the 
classical theta functions given in the introduction. 
The next proposition gives the precise statements of the new forms of the 
square series integral representations satisfied by these theta functions. 

\begin{prop}[Integral Representations for Jacobi Theta Functions] 
\label{prop_JacobiThetaFnVariants_Thetaiq_IntReps_stmts_v3} 
\begin{subequations} 
\label{eqn_JacobiThetaFnVariants_Thetaiq_IntReps_stmts_v3} 
For any $q \in \mathbb{C}$ such that $|q| \in (0, 1/2)$, the 
theta function, $\vartheta_2(q)$, has the integral representation 
\begin{align} 
\label{eqn_JacobiThetaFnVariants_Thetaiq_IntReps-stmt_vT.2(q)} 
\vartheta_2(q) & = 4 q^{1/4} \int_0^{\infty} 
     \frac{e^{-t^2/2}}{\sqrt{2\pi}} \left[ 
     \frac{1 - q \cosh\left(\sqrt{2 \Log(q)} t\right)}{ 
     q^2 - 2q \cosh\left(\sqrt{2 \Log(q)} t\right) + 1} 
     \right] dt. 
\end{align} 
Moreover, 
for any $q \in \mathbb{C}$ such that $|q| \in (0, 1/4)$, the 
remaining two theta function cases, $\vartheta_i(q)$, for $i = 3, 4$, 
satisfy respective integral representations given by 
\begin{align} 
\label{eqn_JacobiThetaFnVariants_Thetaiq_IntReps-stmt_vT.3(q)} 
\vartheta_3(q) & = 1 + 4q \int_0^{\infty} 
     \frac{e^{-t^2/2}}{\sqrt{2\pi}} \left[ 
     \frac{1 - q^2 \cosh\left(\sqrt{2 \Log(q)} t\right)}{ 
     q^4 - 2 q^2 \cosh\left(\sqrt{2 \Log(q)} t\right) + 1} 
     \right] dt \\ 
\label{eqn_JacobiThetaFnVariants_Thetaiq_IntReps-stmt_vT.4(q)} 
\vartheta_4(q) & = 1 - 4q \int_0^{\infty} 
     \frac{e^{-t^2/2}}{\sqrt{2\pi}} \left[ 
     \frac{1 + q^2 \cosh\left(\sqrt{2 \Log(q)} t\right)}{ 
     q^4 + 2 q^2 \cosh\left(\sqrt{2 \Log(q)} t\right) + 1} 
     \right] dt. 
\end{align} 
\end{subequations} 
\end{prop} 
\begin{proof} 
After setting $u := 0$ in the Fourier series forms for the full series 
for the Jacobi theta functions, $\vartheta_i(u, q)$, from the introduction, 
the series for the theta functions, $\vartheta_i(q)$, are expanded for a 
fixed $q$ as follows: 
\begin{align} 
\tag{i} 
\label{eqn_VarThetaiq_JacobiThetaFnSpCases_FourierSeries_defs_stmts_v2} 
\vartheta_2(q) & = 2 q^{1/4} \sum_{n=0}^{\infty} q^{n^2} q^{n} \equiv 
     2 q^{1/4} G_{sq}\left(q, q, 1\right) \\ 
\tag{ii} 
\vartheta_3(q) & = 1 + 2q \sum_{n=0}^{\infty} q^{n^2} q^{2n} \equiv 
     1 + 2q G_{sq}\left(q, q^2, 1\right) \\ 
\tag{iii} 
\vartheta_4(q) & = 1 - 2q \sum_{n=0}^{\infty} q^{n^2} (-1)^{n} q^{2n} \equiv 
     1 - 2q G_{sq}\left(q, q^2, -1\right). 
\end{align} 
The integrals given on the right--hand--sides of 
\eqref{eqn_JacobiThetaFnVariants_Thetaiq_IntReps_stmts_v3} 
follow by applying 
\eqref{eqn_GeomSqSeries_SpFnVariants_GsqpmqcFn_ModIntRep_v2} 
to each of these series first with 
$|q| \in (0, 1/2)$, and 
then with $|q| \in (0, 1/4)$ in the last two cases, respectively. 
\end{proof} 

The next integral formulas for the 
\textit{Mellin transform} involving the Jacobi theta functions, 
$\vartheta_i(0, \imath x^2)$, with $i = 2, 3, 4$ give 
closed-form expansions of special constants defined 
in terms of the \textit{gamma function}, 
$\Gamma(z)$, and the \textit{Riemann zeta function}, $\zeta(s)$, for 
$\Re(s) > 2$ \citep[eq. (20.10.1)--(20.10.3); \S 20.10(i)]{NISTHB}. 

\begin{cor}[Mellin Transforms of the Jacobi Theta Functions] 
\label{cor_JThetaT2_RZeta_MellinTransform_app} 
\label{ex_JThetaT2_RZeta_MellinTransform_app} 
Let $s \in \mathbb{C}$ denote a fixed constant such that 
$\Re(s) > 2$. Then we have that 
\begin{align*} 
(2^s - 1) \pi^{-s/2} & \Gamma(s/2) \zeta(s) = 
     \int_0^{\infty} x^{s-1} \vartheta_2(0, \imath x^2) dx \\ 
   & = 
     \int_0^{\infty} \int_0^{\infty} 
     \frac{e^{-\frac{t^2}{2}-\frac{\pi x^2}{4}}}{\sqrt{2\pi}} 
     \frac{4 x^{s-1} 
     \left(1-e^{-\pi  x^2} \cos \left(\sqrt{2 \pi } t
     x\right)\right)}{\left(e^{-2 \pi  x^2} - 
     2 e^{-\pi  x^2} \cos \left(\sqrt{2 \pi } t x\right)+1\right)} dt dx \\ 
\pi^{-s/2} & \Gamma(s/2) \zeta(s) = \int_0^{\infty} 
     x^{s-1} \left(\vartheta_3(0, \imath x^2) - 1\right) dx \\ 
   & \phantom{\Gamma(s/2) \zeta(s)} = 
     \int_0^{\infty} \int_0^{\infty} 
     \frac{e^{-\frac{t^2}{2} - \pi x^2}}{\sqrt{2\pi}} 
     \frac{4 x^{s-1}
     \left(1-e^{-2 \pi  x^2} \cos \left(\sqrt{2 \pi } t
     x\right)\right)}{\left(e^{-4 \pi  x^2} - 
     2 e^{-2 \pi  x^2} \cos \left(\sqrt{2 \pi } t x\right)+1\right)} dt dx \\ 
(1 - 2^{1-s}) \pi^{-s/2} & \Gamma(s/2) \zeta(s) = \int_0^{\infty} 
     x^{s-1} \left(1 - \vartheta_4(0, \imath x^2)\right) dx \\ 
   & \phantom{\Gamma(s/2) \zeta(s)} = 
     \int_0^{\infty} \int_0^{\infty} 
     \frac{e^{-\frac{t^2}{2}-\pi  x^2}}{\sqrt{2\pi}}
     \frac{4 x^{s-1} \left(1 + e^{-2 \pi  x^2} \cos \left(\sqrt{2 \pi } t
     x\right)\right)}{\left(e^{-4 \pi  x^2} + 
     2 e^{-2 \pi  x^2} \cos \left(\sqrt{2 \pi } t x\right)+1\right)} dt dx 
\end{align*} 
\end{cor} 
\begin{proof} 
The Mellin transform integrals given in terms of the Jacobi theta function 
variants of the form 
$\vartheta_i(0 \mid \imath x^2) = \vartheta_i\left(e^{-\pi x^2}\right)$ 
are stated in the results of \citep[\S 20.10(i)]{NISTHB}. 
The double integral results follow by applying the integral 
representations for the Jacobi theta functions given by 
Proposition \ref{prop_JacobiThetaFnVariants_Thetaiq_IntReps_stmts_v3}. 
\end{proof} 

\subsubsection*{Explicit Values of Ramanujan's 
                $\varphi$--Function and $\psi$--Function} 
\label{subsubSection_Apps_SpFnsRelated_to_GeomSqSeries_RamPsiPhiFns} 
\label{def_RamPsiFn_and_RamPhiFn_series_defs} 

The series for \keywordemph{Ramanujan's functions}, 
$\varphi(q)$ and $\psi(q)$, are expanded similarly 
through the series for the classical theta functions as 
$\varphi(q) \equiv \vartheta_3(0, q)$ and 
$\psi(q) \equiv \left(2 q^{1/8}\right)^{-1} \times 
 \vartheta_2\left(0, q^{1/2}\right)$ 
\citep[\cf \S XII]{RAMANUJAN} \citep[\cf \S 16, Entry 22]{BERNDTRAMNBIII}.
These functions also define integral representations for 
classes of special explicit constants stated in the results of the next 
corollaries below. 

%% See : "temp-working-2013.03.26-v1.*": 
\begin{cor}[Integral Representations] 
\label{cor_FirstSqSeriesExs_RamPsiFn_LLambdaLambertSeries} 
Suppose that $q \in \mathbb{C}$ and $|q| < 1$. 
For these inputs of $q$, \emph{Ramanujan's functions}, 
$\varphi(q)$ and $\psi(q)$, respectively, 
have the following integral representations: 
\begin{align*} 
\varphi(q) & = 1 + \int_0^{\infty} \frac{e^{-t^2/2}}{\sqrt{2\pi}} 
     \left[\frac{4q \left(1-q^2 \cosh\left( 
     \sqrt{2 \Log(q)} t\right)\right)}{q^4-2 q^2
     \cosh\left(\sqrt{2 \Log(q)} t\right) + 1} 
     \right] dt \\ 
\psi(q) & = \int_0^{\infty} 
     \frac{2 e^{-t^2/2}}{\sqrt{2\pi}} 
     \left[\frac{\left(1-\sqrt{q} 
     \cosh\left(\sqrt{\Log(q)} t\right)\right)}{q-2 \sqrt{q} 
     \cosh\left(\sqrt{\Log(q)} t\right) + 1} 
     \right] dt. 
\end{align*} 
\end{cor} 
\begin{proof} 
On the stated region where $|q| < 1$, these two special functions have 
series expansions in the form of 
\eqref{eqn_GeomSqSeries_SpFnVariants_GsqpmqcFn_ModIntRep_v2} given by 
\begin{align*} 
\varphi(q) & = 
     1 + 2q \sum_{n=0}^{\infty} q^{n^2} q^{2n} \ \equiv 
     1 + 2q G_{\sq}\left(q, q^2, 1\right) \\ 
\psi(q) & = \sum_{n=0}^{\infty} \left(\sqrt{q}\right)^{n^2} 
     \left(\sqrt{q}\right)^{n} \equiv 
     G_{\sq}\left(\sqrt{q}, \sqrt{q}, 1\right). 
\end{align*} 
The new integral representations for these functions then follow from 
Proposition \ref{prop_OrdExp_GeomSquareSeries} 
applied to the series on the right-hand-side of each of the 
previous equations. 
\end{proof} 

\begin{cor}[Special Values of Ramanujan's $\varphi$-Function] 
\label{cor_VarPhi_JTheta3_Integrals_for_SpValues} 
%% See : "temp-working-2014.01.31-v1.*": 
For any $k \in \mathbb{R}^{+}$, the variant of the 
\emph{Ramanujan $\varphi$-function}, 
$\varphi\left(e^{-k\pi}\right) \equiv \vartheta_3\left(e^{-k\pi}\right)$, 
has the integral representation 
\begin{align} 
\label{eqn_RamVarPhiFn_JTheta3_SpValues_EmkPi_gen_integral} 
\varphi\left(e^{-k\pi}\right) & = 
     1 + \int_0^{\infty} \frac{e^{-t^2/2}}{\sqrt{2\pi}} \left[ 
     \frac{4 e^{k\pi} \left(e^{2k\pi} - \cos\left(\sqrt{2\pi k} t\right) 
     \right)}{e^{4k\pi} - 2 e^{2k\pi} \cos\left(\sqrt{2\pi k} t\right) + 1} 
     \right] dt. 
\end{align} 
Moreover, the 
special values of this function corresponding to the 
particular cases of $k \in \{1,2,3,5\}$ in 
\eqref{eqn_RamVarPhiFn_JTheta3_SpValues_EmkPi_gen_integral} 
have the respective integral representations 
\begin{align} 
\label{eqn_RamVarPhiFn_JTheta3_SpValues_EmkPi_kEQ1235_integrals} 
\frac{\pi^{1/4}}{\Gamma\left(\frac{3}{4}\right)} & = 
     1 + \int_0^{\infty} \frac{e^{-t^2/2}}{\sqrt{2\pi}} \left[ 
     \frac{4 e^{\pi} \left(e^{2\pi} - \cos\left(\sqrt{2\pi} t\right) 
     \right)}{e^{4\pi} - 2 e^{2\pi} \cos\left(\sqrt{2\pi} t\right) + 1} 
     \right] dt \\ 
\notag 
\frac{\pi^{1/4}}{\Gamma\left(\frac{3}{4}\right)} \cdot 
     \frac{\sqrt{\sqrt{2} + 2}}{2} & = 
     1 + \int_0^{\infty} \frac{e^{-t^2/2}}{\sqrt{2\pi}} \left[ 
     \frac{4 e^{2\pi} \left(e^{4\pi} - \cos\left(2 \sqrt{\pi} t\right) 
     \right)}{e^{8\pi} - 2 e^{4\pi} \cos\left(2 \sqrt{\pi} t\right) + 1} 
     \right] dt \\ 
\notag 
\frac{\pi^{1/4}}{\Gamma\left(\frac{3}{4}\right)} \cdot 
     \frac{\sqrt{\sqrt{3} + 1}}{2^{1/4} 3^{3/8}} & = 
     1 + \int_0^{\infty} \frac{e^{-t^2/2}}{\sqrt{2\pi}} \left[ 
     \frac{4 e^{3\pi} \left(e^{6\pi} - \cos\left(\sqrt{6 \pi} t\right) 
     \right)}{e^{12\pi} - 2 e^{6\pi} \cos\left(\sqrt{6 \pi} t\right) + 1} 
     \right] dt \\ 
\notag 
\frac{\pi^{1/4}}{\Gamma\left(\frac{3}{4}\right)} \cdot 
     \frac{\sqrt{5 + 2 \sqrt{5}}}{5^{3/4}} & = 
     1 + \int_0^{\infty} \frac{e^{-t^2/2}}{\sqrt{2\pi}} \left[ 
     \frac{4 e^{5\pi} \left(e^{10\pi} - \cos\left(\sqrt{10 \pi} t\right) 
     \right)}{e^{20\pi} - 2 e^{10\pi} \cos\left(\sqrt{10 \pi} t\right) + 1} 
     \right] dt. 
\end{align} 
\end{cor} 
\begin{proof} 
The first integral in 
\eqref{eqn_RamVarPhiFn_JTheta3_SpValues_EmkPi_gen_integral} is obtained from 
Corollary \ref{cor_FirstSqSeriesExs_RamPsiFn_LLambdaLambertSeries} 
applied to the functions 
$\varphi\left(q\right) \equiv \vartheta_3\left(q\right)$ at 
$q := e^{-k\pi}$. 
The constants on the left-hand-side of the integral equations in 
\eqref{eqn_RamVarPhiFn_JTheta3_SpValues_EmkPi_kEQ1235_integrals} 
correspond to the known values of Ramanujan's function, 
$\varphi\left(e^{-k\pi}\right)$, over the particular inputs of 
$k \in \{1,2,3,5\}$, in respective order. 
These explicit formulas for the values of 
$\varphi\left(e^{-k\pi}\right)$ 
are established by Theorem 5.5 of \citep{THETAFN-IDENTS-EXPLICIT-FORMULAS}. 
\end{proof} 

The special cases of the right-hand-side integrals 
in the corollary 
are verified numerically to match the constant values cited by 
each of the equations in 
\eqref{eqn_RamVarPhiFn_JTheta3_SpValues_EmkPi_kEQ1235_integrals} in 
\citep{SQS-SUMMARYNB-REF}. 
Still other integral formulas for the values of the function, 
$\varphi\left(e^{-k\pi}\right)$, are known in terms of formulas 
involving powers of positive rational inputs to the gamma function, 
 in the form of \eqref{eqn_RamVarPhiFn_JTheta3_SpValues_EmkPi_gen_integral}, 
for example, when $k \in \{\sqrt{2},\sqrt{3},\sqrt{6}\}$. 

\begin{cor}[Explicit Values of Ramanujan's $\psi$-Function] 
\label{cor_RamPsi_JTheta2_Integrals_for_SpValues} 
%% See : "temp-working-2014.01.31-v1.*": 
For any $k \in \mathbb{R}^{+}$, the forms of the 
\emph{Ramanujan $\psi$--function}, 
$\psi\left(e^{-k\pi}\right) \equiv 
 \frac{1}{2} e^{k\pi / 8}\vartheta_2\left(e^{-k\pi / 2}\right)$, 
have the integral representation 
\begin{align} 
\label{eqn_RamVarPsiFn_JTheta2_SpValues_EmkPi_gen_integral} 
\psi\left(e^{-k\pi}\right) & = 
     \int_0^{\infty} \frac{e^{-t^2/2}}{\sqrt{2\pi}} \left[ 
     \frac{\cos\left(\sqrt{k \pi} t\right) - e^{k\pi / 2}}{ 
     \cos\left(\sqrt{k \pi} t\right) - \cosh\left(\frac{k\pi}{2}\right)} 
     \right] dt. 
\end{align} 
The explicit values of this function corresponding to the 
choices of $k \in \{1,2,1/2\}$ in 
\eqref{eqn_RamVarPsiFn_JTheta2_SpValues_EmkPi_gen_integral} 
have the following respective integral representations: 
\begin{align} 
\label{eqn_RamVarPsiFn_JTheta2_SpValues_EmkPi_kEQ1235_integrals} 
\frac{\pi^{1/4}}{\Gamma\left(\frac{3}{4}\right)} \cdot 
     \frac{e^{\pi / 8}}{2^{5/8}} & = 
     \int_0^{\infty} \frac{e^{-t^2/2}}{\sqrt{2\pi}} \left[ 
     \frac{\cos\left(\sqrt{\pi} t\right) - e^{\pi / 2}}{ 
     \cos\left(\sqrt{\pi} t\right) - \cosh\left(\frac{\pi}{2}\right)} 
     \right] dt \\ 
\notag 
\frac{\pi^{1/4}}{\Gamma\left(\frac{3}{4}\right)} \cdot 
     \frac{e^{\pi / 4}}{2^{5/4}} & = 
     \int_0^{\infty} \frac{e^{-t^2/2}}{\sqrt{2\pi}} \left[ 
     \frac{\cos\left(\sqrt{2 \pi} t\right) - e^{\pi}}{ 
     \cos\left(\sqrt{2 \pi} t\right) - \cosh\left(\pi\right)} 
     \right] dt \\ 
\notag 
\frac{\pi^{1/4}}{\Gamma\left(\frac{3}{4}\right)} \cdot 
     \frac{\left(\sqrt{2} + 1\right)^{1/4} e^{\pi / 16}}{2^{7/16}} & = 
     \int_0^{\infty} \frac{e^{-t^2/2}}{\sqrt{2\pi}} \left[ 
     \frac{\cos\left(\sqrt{\frac{\pi}{2}} t\right) - e^{\pi / 4}}{ 
     \cos\left(\sqrt{\frac{\pi}{2}} t\right) - \cosh\left(\frac{\pi}{4}\right)} 
     \right] dt. 
\end{align} 
\end{cor} 
\begin{proof} 
The integral representation in 
\eqref{eqn_RamVarPsiFn_JTheta2_SpValues_EmkPi_gen_integral} is obtained from 
Corollary \ref{cor_FirstSqSeriesExs_RamPsiFn_LLambdaLambertSeries} 
at the inputs of $q := e^{-k\pi}$. 
The explicit formulas for constants on the 
left-hand-side of the integral equations in 
\eqref{eqn_RamVarPsiFn_JTheta2_SpValues_EmkPi_kEQ1235_integrals} 
are the known values of Ramanujan's function, $\psi\left(e^{-k\pi}\right)$, 
from Theorems 6.8 and 6.9 of the reference \citep{RAMTHETAFNS-EXPLICIT-VALUES} 
corresponding to these particular values of $k \in \mathbb{Q}^{+}$. 
\end{proof} 

\subsection{Integral Formulas for Sequences Involving Powers of Linear Polynomials} 
\label{subsubSection_SqSeries_for_PolyPowers_DerivsOfJThetaFns} 

The primary applications motivating the next results proved in this 
section correspond to the higher-order derivatives of the Jacobi theta functions 
defined by $d^{(j)} \vartheta_i(u, q) / du^{(j)} \mid_{u=u_0}$ 
for some prescribed setting of $u_0$, 
including the special cases of the theta functions from the 
results for these series from 
\sref{subsubSection_Apps_SpCases_of_the_JacobiThetaFns} 
where $u_0 = 0$. 
The square series integral representations for polynomial powers, and then 
more generally for any fixed linear polynomial multiple, of the 
geometric square series are easily obtained by applying 
Proposition \ref{prop_S2_OGFFz_to_nPowmOGFFz_stmt}
to the forms of $G_{\sq}(q, c, z)$ with respect to the underlying series in $z$. 

\subsubsection{Initial Results} 
\label{subsubSection_Cors_for_PolyPowersOf_SpCase_Geometric_Series} 

\begin{prop}[Integrals for Polynomial Powers] 
\label{prop_IntsPolyPows_OGF_pow_int-stmts_form_v1} 
Suppose that $m \in \mathbb{N}$ and $q, c, z \in \mathbb{C}$ 
are defined such that $|q| \in (0, 1)$ and where $|cz| < 1$. 
Then the square series function, $\vartheta_{0,m}(q, c, z)$, 
defined by the series in 
\eqref{eqn_Theta_dm_cqz_ShiftedPolyPowers_GeomSeries-Based_series_stmt_v1} 
satisfies an integral formula given by 
\begin{align} 
\label{eqn_prop_IntsPolyPows_OGF_pow_int_stmt_v1} 
 & \vartheta_{0,m}(q, c, z) = 
     \int_0^{\infty} \frac{e^{-t^2/2}}{\sqrt{2\pi}} \left[ 
     \sum_{k=0}^{m} \gkpSII{m}{k} 
     \frac{(cz)^k k! \times \NumFn_k\left(\sqrt{2 \Log(q)} t, cz\right)}{ 
     \left(c^2z^2 - 2cz \cosh\left(\sqrt{2 \Log(q)} t\right) + 1 
     \right)^{k+1}} 
     \right] dt. 
\end{align} 
The numerator terms 
of the integrands in the previous equation denote the 
polynomials $\NumFn_k(w, z) \in \mathbb{C}\lbrack z \rbrack$ of 
degree $k+1$ in $z$ defined by the following equation 
for non-negative integers $k \leq m$: 
\begin{align} 
\notag 
\NumFn_k(w, cz) & := 
     \sum_{b = \pm 1} e^{-bk w} \times 
     \left(1 - e^{b w} cz\right)^{k+1}. 
\end{align} 
\end{prop} 
\begin{proof} 
\label{cor_OrdExp_GeomSquareSeries_IntRep_for_PartialFraction_exps_stmts_v2} 
\label{example_OrdExp_GeomSquareSeries} 
By the proof of Theorem \ref{thm_SqSeries_OGF_Transforms}, 
the integral representation of the series in the first of the next 
equations in \eqref{eqn_GeomSqSeries_Gqcz_partial_frac_exp_IntRep_stmt_v2} 
can be differentiated termwise with respect to $z$ to arrive at the 
$j^{th}$ derivative formulas given in 
\eqref{eqn_GeomSqSeries_Gqcz_ithDerivs_partial_frac_exp_IntRep_stmt_v3}. 
\begin{align} 
\label{eqn_GeomSqSeries_Gqcz_partial_frac_exp_IntRep_stmt_v2} 
\tag{i} 
\vartheta_{0,0}(q, c, z) & = \int_0^{\infty} \frac{e^{-t^2/2}}{\sqrt{2\pi}} \left[ 
     \frac{1}{\left(1 - e^{\sqrt{2 \Log(q)} t} cz\right)} + 
     \frac{1}{\left(1 - e^{-\sqrt{2 \Log(q)} t} cz\right)} 
     \right] dt \\ 
\label{eqn_GeomSqSeries_Gqcz_ithDerivs_partial_frac_exp_IntRep_stmt_v3} 
\vartheta_{0,m}^{(j)}(q, c, z) & = 
     \int_0^{\infty} \frac{e^{-t^2/2}}{\sqrt{2\pi}} \left[ 
     \frac{e^{j \sqrt{2 \Log(q)} t} c^j \cdot j!}{ 
     \left(1 - e^{\sqrt{2 \Log(q)} t} cz\right)^{j+1}} + 
     \frac{e^{-j \sqrt{2 \Log(q)} t} c^j \cdot j!}{ 
     \left(1 - e^{-\sqrt{2 \Log(q)} t} cz\right)^{j+1}} 
     \right] dt \\ 
\notag 
     & = 
     \int_0^{\infty} \frac{e^{-t^2/2}}{\sqrt{2\pi}} \left[ 
     \frac{c^j j! \sum_{b = \pm 1} e^{bj \sqrt{2 \Log(q)} t} \times 
     \left(1 - e^{-b\sqrt{2 \Log(q)} t} cz\right)^{j+1}}{ 
     \left(c^2z^2 - 2cz \cosh\left(\sqrt{2 \Log(q)} t\right) + 1\right)^{j+1}} 
     \right] dt. 
\end{align} 
We may then apply 
Proposition \ref{prop_S2_OGFFz_to_nPowmOGFFz_stmt} and the 
definition of the numerator terms, $\NumFn_k(w, cz)$, to the previous 
equations to obtain the result stated in 
\eqref{eqn_prop_IntsPolyPows_OGF_pow_int_stmt_v1}. 
\end{proof} 

%% See near end of "temp-working-2014.06.21-v1.*": 
\begin{table}[h] 

\begin{center} 

\setlength{\fboxrule}{1.25pt} 
\setlength{\fboxsep}{4pt} 
\fbox{\begin{minipage}{\textwidth} 

     \tabulinestyle{2pt OTblBC} 
     \taburulecolor |OTblBC|{ITblBC} 
     \arrayrulewidth=1.5pt \doublerulesep=0pt \tabulinesep=3pt%\extrarowsep=3pt
     %\begin{tabu} to \TblW {|c |T{+1}{c}|T{+1}{c}|T{+1}{c}|X[l]|} \hline 

     \begin{tabu} to \TblW  {|c|c|X[l]|}

     \taburowcolors 2{Plum!32!White .. Plum!10!White} 
     \hline 
     \multicolumn3{|l|}{ 
          \bfseries\normalsize{Scaled Formulas for the Numerator Functions} 
     } \\ 
     \hline\hline  
     \rowfont{\bfseries\small\mathversion{bold}} 
     \textbf{Eq.} & $\mathbf{k}$ & 
     \textbf{Scaled Numerator Function} 
     ($\frac{1}{2} \times \NumFn_k(s, y)$) \\ 
     \taburowcolors 2{Plum!0!White .. Plum!0!White} 
     \hline\hline

     \LabelTableEqn{eqn_Table-ScaledPolyPowerSnumFns-v_kEQ0} & 
     $0$ & 
     $\phantom{-} 1 -y \cosh(s)$ \\ 
     \hline 
     
     \LabelTableEqn{eqn_Table-ScaledPolyPowerSnumFns-v_kEQ01} & 
     $1$ & 
     $-2y + (y^2+1) \cosh(s)$ \\ 
     \hline 
     
     \LabelTableEqn{eqn_Table-ScaledPolyPowerSnumFns-v_kEQ02} & 
     $2$ & 
     $\phantom{-} 3y^2 - (y^3+3y) \cosh(s) + \cosh(2s)$ \\ 
     \hline 
     
     \LabelTableEqn{eqn_Table-ScaledPolyPowerSnumFns-v_kEQ03} & 
     $3$ & 
     $-4 y^3 + (y^4+6y^2) \cosh(s) - 4y \cosh(2s) + \cosh(3s)$ \\ 
     \hline 
     
     \LabelTableEqn{eqn_Table-ScaledPolyPowerSnumFns-v_kEQ04} & 
     $4$ & 
     $\phantom{-} 5 y^4 - (y^5+10y^3) \cosh(s) + 10 y^2 \cosh(2s) 
      -5y \cosh(3s) + \cosh(4s)$ \\ 

     \hline\hline 

     \end{tabu} 

     \setcaption{table}{ 
          Several Formulas for the Numerator Functions, $\NumFn_k(s, y)$}{ 
          Explicit formulas for the numerator functions, 
          $\NumFn_k(s, y)$, from the definition given in 
          Proposition \ref{prop_IntsPolyPows_OGF_pow_int-stmts_form_v1} 
          over the first several special cases of $k \in [0, 4]$. 
          } 
     \label{table_SeveralFormulas_for_NumFns_Numksy} 

\end{minipage}} 
\end{center} 

\end{table}

Proposition \ref{prop_IntsPolyPows_OGF_pow_int-stmts_form_v1} 
provides an integral representation of the polynomial multiples for any 
geometric series sequences of the form $f_n := p(n) \times c^n$ 
corresponding to any fixed $p(n) \in \mathbb{C}\lbrack n \rbrack$. 
Notice that 
for polynomials defined as integral powers of the form 
$p(n) := (\alpha n+\beta)^{m}$ for some fixed $m \in \mathbb{Z}^{+}$, 
for example, as in the series for the $m^{th}$ derivatives of the 
theta functions, $\vartheta_i^{(m)}(q)$, the 
repeated terms over the index $k \in \{0, 1, \ldots, m\}$ 
in the inner sums from 
\eqref{eqn_prop_IntsPolyPows_OGF_pow_int_stmt_v1} 
may be simplified even further to obtain the next simplified formulas in the 
slightly more general results for these series expansions. 
Table \ref{table_SeveralFormulas_for_NumFns_Numksy} 
lists several simplified expansions of the functions, $\NumFn_k(w, z)$, 
defined in the statement of 
Proposition \ref{prop_IntsPolyPows_OGF_pow_int-stmts_form_v1}. 

\begin{prop} 
\label{prop_IntsPolyPowsMultiples_OGF_pow_int-stmts_form_v2} 
Suppose that $m \in \mathbb{N}$ and $\alpha, \beta, q, c, z \in \mathbb{C}$ 
are defined such that $|q| \in (0, 1)$ and such that $|cz| < 1$. 
The modified square series from 
\eqref{eqn_Theta_dm_cqz_ShiftedPolyMultiples_GeomSeries-Based_series_stmt_v2} 
satisfies an integral representation of the form 
\begin{align} 
\label{eqn_prop_IntsPolyPowsMults_OGF_pow_int_stmt_v2} 
 & \vartheta_{0,m}(\alpha, \beta; q, c, z) \\ \notag & = 
     \int_0^{\infty} \frac{e^{-t^2/2}}{\sqrt{2\pi}} \left[ 
     \sum_{0 \leq i \leq k \leq m} 
     \binom{k}{i} 
     \frac{(-1)^{k-i} (\alpha i+\beta)^{m} (cz)^{k} 
     \NumFn_k\left(\sqrt{2 \Log(q)} t, cz\right)}{ 
     \left(c^2 z^2 - 2cz \cosh\left(\sqrt{2 \Log(q)} t\right) + 
     1\right)^{k+1}} 
     \right] dt. 
\end{align} 
\end{prop} 
\begin{proof} 
The proof is similar to the proof of 
Proposition \ref{prop_IntsPolyPows_OGF_pow_int-stmts_form_v1}. 
We employ the same notation and $j^{th}$ derivative expansions from the 
proof of the first proposition. 
To simplify notation, for $j \geq 0$ let 
\begin{align*} 
N_j(t, q, cz) = \frac{(cz)^j \NumFn_j\left(\sqrt{2 \Log(q)} t, cz\right)}{ 
     \left(c^2 z^2 - 2cz \cosh\left(\sqrt{2 \Log(q)} t\right) + 
     1\right)^{j+1}}. 
\end{align*}
Since the Stirling numbers of the second kind are expanded by the 
finite sum formula \citep[\S 26.8(i)]{NISTHB} 
\begin{equation*} 
\gkpSII{i}{k} = \frac{1}{k!} \sum_{j=0}^{k} \binom{k}{j} (-1)^{k-j} j^i, 
\end{equation*} 
by expanding out the polynomial powers of $(\alpha n+\beta)^m$ 
according to the binomial theorem we arrive at the following expansions of 
Proposition \ref{prop_S2_OGFFz_to_nPowmOGFFz_stmt}: 
\begin{align*} 
\vartheta_{0,m}(\alpha, \beta; q, c, z) & = 
     \int_0^{\infty} \frac{e^{-t^2/2}}{\sqrt{2\pi}} \times 
     \sum_{j=0}^{m} \sum_{k=0}^{j} 
     \binom{m}{j} \gkpSII{j}{k} \alpha^j \beta^{m-j} \times k! N_k(t, q, cz) dt \\ 
     & = 
     \int_0^{\infty} \frac{e^{-t^2/2}}{\sqrt{2\pi}} \times 
     \sum_{j=0}^{m} \sum_{k=0}^{j} \sum_{i=0}^{k} 
     \binom{m}{j} \binom{k}{i} (\alpha i)^{j} \beta^{m-j} (-1)^{k-i} 
     N_k(t, q, cz) dt \\ 
     & = 
     \int_0^{\infty} \frac{e^{-t^2/2}}{\sqrt{2\pi}} \times 
     \sum_{k=0}^{j} \sum_{i=0}^{k} 
     \binom{k}{i} (\alpha i + \beta)^{m} (-1)^{k-i} N_k(t, q, cz) dt. 
     \qedhere 
\end{align*} 
\end{proof} 

\subsubsection{Higher-Order Derivatives of the Jacobi Theta Functions} 

\begin{example}[Special Case Derivative Series] 
\label{example_JThetaFns_iEqOne_Series_for_SpCaseDerivs} 
An integral representation for the 
special case of the first derivative of the theta function, 
$\vartheta_1^{\prime}(q) \equiv 2q^{1/4} \cdot Q_{2,1}(q, q, -1)$, 
expanded as the series 
%% : See "-2014.02.27-v1.*": 
\begin{align} 
\notag 
\vartheta_1^{\prime}(q) & = 2 q^{1/4} \sum_{n=0}^{\infty} 
     q^{n^2} (2n+1) (-q)^{n}, 
\end{align} 
is a consequence of the related statement in  
Proposition \ref{prop_Qsq_abcqz_unilateral_series_fn_integral_rep_v1} 
already cited by the results given above. 
This series arises in expanding the cube powers of the infinite 
$q$-Pochhammer symbol, $(q)_{\infty}^3$, cited as an example in 
Section \ref{subSection_Intro_Examples}. 

The higher-order series for the third derivative of this theta function 
provides another example of the new generalization of this first 
result stated by 
Proposition \ref{prop_IntsPolyPowsMultiples_OGF_pow_int-stmts_form_v2} 
expanded as in the following equations for $|q| \in (0, 1/2)$: 
\begin{align} 
\notag 
\vartheta_1^{\prime\prime\prime}(q) & = 
     \left(2 q^{1/4}\right) \times 
     \sum_{n=0}^{\infty} q^{n^2} (2n+1)^{3} (-q)^{n} \\ 
\notag 
   & = 
     \int_0^{\infty} \frac{2 e^{-t^2/2}}{\sqrt{2\pi}} \left[ 
     \sum_{0 \leq i \leq k \leq 3} 
     \binom{k}{i} 
     \frac{(-1)^{i} (2i+1)^{3} q^{k+1/4} 
     \NumFn_k\left(\sqrt{2 \Log(2)} t, -q\right)}{ 
     \left(q^{2} + 2 q \cosh\left(\sqrt{2 \Log(q)} t\right) + 1\right)^{k+1}} 
     \right] dt. 
\end{align} 
\end{example} 
Notice that the 
subsequent cases of the higher-order, $j^{th}$ derivatives of 
these theta functions are then formed from the Fourier series 
expansions of the classical functions, 
$\vartheta_i(u, q)$, expanded by the series in 
\eqref{eqn_JacobiTheta_fn_T.1} through \eqref{eqn_JacobiTheta_fn_T.4}. 
More precisely, 
if $j \in \mathbb{Z}^{+}$, the 
new integral representations for the 
higher--order $j^{th}$ derivatives of the functions, 
$\vartheta_i^{(j)}(q) \equiv \vartheta_i^{(j)}(0, q)$, 
with respect to their first parameter 
are then obtained from these series through 
Proposition \ref{prop_IntsPolyPowsMultiples_OGF_pow_int-stmts_form_v2} 
over odd-ordered positive integers $j := 2m+1$ when $i = 1$, and 
at even-ordered non-negative cases of $j := 2m$ when $i = 2,3,4$ 
for any $m \in \mathbb{N}$. 
Lemma \ref{lemma_JThetaFns_GenSeries_for_Higher-OrderDerivs-stmts_v1} and 
Corollary \ref{cor_JThetaFns_GenSeries_for_Higher-OrderDerivs-stmts_v2} 
given below state the exact series formulas and corresponding 
integral representations for the higher-order derivatives of these 
particular variants of the Jacobi theta functions. 

\begin{lemma}[Generalized Series for Higher-Order Derivatives] 
\label{lemma_JThetaFns_GenSeries_for_Higher-OrderDerivs-stmts_v1} 
For any non-negative $j \in \mathbb{Z}$, the higher-order 
$j^{th}$ derivatives, $\vartheta_i^{(j)}(0, q)$, of the classical 
Jacobi theta functions, $\vartheta_i(u, q)$, with respect to $u$ 
are expanded through the unilateral power series for the functions defined by 
\eqref{eqn_Theta_dm_cqz_ShiftedPolyMultiples_GeomSeries-Based_series_stmt_v2}
as 
\begin{align} 
\label{eqn_Formulas_for_JThetaFns_Higher-OrderDerivs-stmts_jthderivs_v1} 
\vartheta_1^{(j)}(q) & = \left(2q^{1/4}\right) \times 
     \vartheta_{0,j}\left(2, 1; q, q, -1\right) 
     \Iverson{j \equiv 1 \pmod{2}} \\ 
\notag 
\vartheta_2^{(j)}(q) & = \left(2q^{1/4}\right) \times 
     \vartheta_{0,j}\left(2, 1; q, q, -1\right) 
     \Iverson{j \equiv 0 \pmod{2}} \\ 
\notag 
\vartheta_3^{(j)}(q) & = 
     \Iverson{j = 0} + \left(2 q\right) \times 
     \vartheta_{0,j}\left(2, 2; q, q^{2}, 1\right) 
     \Iverson{j \equiv 0 \pmod{2}} \\ 
\notag 
\vartheta_4^{(j)}(q) & = 
     \Iverson{j = 0} - \left(2 q\right) \times 
     \vartheta_{0,j}\left(2, 2; q, q^{2}, -1\right) 
     \Iverson{j \equiv 0 \pmod{2}}, 
\end{align} 
for $q \in \mathbb{C}$ satisfying some 
$|q| \in \left(0, R_q\left(\vartheta_i\right)\right)$ such that the 
upper bound, $R_q\left(\vartheta_i\right)$, on the interval is defined as 
$R_q\left(\vartheta_{12}\right) \equiv 1/2$ when $i := 1,2$, and as 
$R_q\left(\vartheta_{34}\right) \equiv 1/4$ when $i := 3,4$. 
\end{lemma} 
\begin{proof} 
To prove the results for the special case series expanded by 
\eqref{eqn_Formulas_for_JThetaFns_Higher-OrderDerivs-stmts_jthderivs_v1}, 
first observe that for any $j \in \mathbb{N}$, the 
higher-order derivatives of the next Fourier series with respect to $u$ satisfy 
\begin{align} 
\label{eqn_proof_lemma_JThetaFns_GenSeries_for_Higher-OrderDerivs-v1-tag_i} 
\tag{i} 
 & \frac{\partial^{(j)}}{\partial{u}^{(j)}}\left[ 
     1 + c_0(q) \times \sum_{n=0}^{\infty} 
     q^{n^2} \scFn\left((\alpha n+\beta) \cdot u\right) c^n z^n 
     \right]_{u=0} \\ 
\notag 
   & \phantom{\frac{\partial^{(j)}}{\partial{u}^{(j)}}\Biggl[ } = 
     \Iverson{j = 0} + 
     c_0(q) \times \sum_{n=0}^{\infty} \left[ 
     (\alpha n+\beta)^{j} \cdot \scFn^{(j)}(0) 
     \right] \times 
     q^{n^2} (cz)^n, 
\end{align} 
for some function $c_0(q)$ that does not depend on $u$, 
series parameters $\alpha, \beta \in \mathbb{R}$, and where the 
derivatives of the trigonometric functions, denoted in 
shorthand by $\scFn \in \{\sin, \cos\}$, 
correspond to the known formulas for these functions from calculus.
The theta functions, $\vartheta_i(q)$ and $\vartheta_i^{(j)}(q)$, on the 
left-hand-side of 
\eqref{eqn_Formulas_for_JThetaFns_Higher-OrderDerivs-stmts_jthderivs_v1} 
form special cases of the Fourier series for the 
Jacobi theta functions, $\vartheta_i(u, q)$, that are then expanded by 
\eqref{eqn_proof_lemma_JThetaFns_GenSeries_for_Higher-OrderDerivs-v1-tag_i} 
where $\cos(0) = 1$ and $\sin(0) = 0$. 
\end{proof} 

\begin{cor}[Integrals for Higher-Order Derivatives of the Jacobi Theta Functions] 
\label{cor_JThetaFns_GenSeries_for_Higher-OrderDerivs-stmts_v2} 
Let $q \in \mathbb{C}$ be defined on some interval 
$|q| \in \left(0, R_q\left(\vartheta_i\right)\right)$ where 
$R_q\left(\vartheta_{12}\right) \equiv 1/2$ when $i := 1,2$, and as 
$R_q\left(\vartheta_{34}\right) \equiv 1/4$ when $i := 3,4$. 
Then for any such $q$ and fixed $m \in \mathbb{Z}^{+}$, the 
higher-order derivatives of the theta functions, 
$\vartheta_i(q)$, satisfy each of the following integral representations: 
\begin{align} 
\label{eqn_Formulas_for_JThetaFns_Higher-OrderDerivs-stmts_jthderivs_v2} 
\vartheta_1^{(2m+1)}(q) & = 
     \int_{0}^{\infty} \frac{2^{} e^{-t^2/2}}{\sqrt{\pi}} 
     \Biggl[ 
     \sum_{0 \leq i \leq k \leq 2m+1} 
     \binom{k}{i} (-1)^{i} (2i+1)^{2m+1} q^{k+1/4} \times \\ 
\notag 
   & \phantom{= \int_{0}^{\infty} \frac{2^{} e^{-t^2/2}}{\sqrt{\pi}}\Biggl[} 
     \times 
     \frac{\NumFn_k\left(\sqrt{2 \Log(q)} t, -q\right)}{ 
     \left(q^{2} + 2 q \cosh\left(\sqrt{2 \Log(q)} t\right) + 1\right)^{k+1}} 
     \Biggr] dt \\ 
\notag 
\vartheta_2^{(2m)}(q) & = 
     \int_{0}^{\infty} \frac{2^{} e^{-t^2/2}}{\sqrt{\pi}} 
     \Biggl[ 
     \sum_{0 \leq i \leq k \leq 2m} 
     \binom{k}{i} (-1)^{i} (2i+1)^{2m} q^{k+1/4} \times \\ 
\notag 
   & \phantom{= \int_{0}^{\infty} \frac{2^{} e^{-t^2/2}}{\sqrt{\pi}}\Biggl[} 
     \times 
     \frac{\NumFn_k\left(\sqrt{2 \Log(q)} t, -q\right)}{ 
     \left(q^{2} + 2 q \cosh\left(\sqrt{2 \Log(q)} t\right) + 1\right)^{k+1}} 
     \Biggr] dt \\ 
\notag 
\vartheta_3^{(2m)}(q) & = 
     \int_{0}^{\infty} \frac{2^{2m+1} e^{-t^2/2}}{\sqrt{\pi}} 
     \Biggl[ 
     \sum_{0 \leq i \leq k \leq 2m} 
     \binom{k}{i} (-1)^{k-i} (i+1)^{2m} q^{2k+1} \times \\ 
\notag 
   & \phantom{= \Iverson{m = 0} + 
              \int_{0}^{\infty} \Biggl[} \times 
     \frac{\NumFn_k\left(\sqrt{2 \Log(q)} t, q^2\right)}{ 
     \left(q^{4} - 2 q^{2} \cosh\left(\sqrt{2 \Log(q)} t\right) + 1\right)^{k+1}} 
     \Biggr] dt \\ 
\notag 
\vartheta_4^{(2m)}(q) & = 
     \int_{0}^{\infty} \frac{2^{2m+1} e^{-t^2/2}}{\sqrt{\pi}} 
     \Biggl[ 
     \sum_{0 \leq i \leq k \leq 2m} 
     \binom{k}{i} (-1)^{i+1} (i+1)^{2m} q^{2k+1} \times \\ 
\notag 
   & \phantom{= \Iverson{m = 0} + 
              \int_{0}^{\infty} \Biggl[} \times 
     \frac{\NumFn_k\left(\sqrt{2 \Log(q)} t, -q^2\right)}{ 
     \left(q^{4} + 2 q^{2} \cosh\left(\sqrt{2 \Log(q)} t\right) + 
     1\right)^{k+1}} 
     \Biggr] dt. 
\end{align} 
\end{cor} 
\begin{proof} 
The formulas in 
\eqref{eqn_Formulas_for_JThetaFns_Higher-OrderDerivs-stmts_jthderivs_v2} 
follow immediately as consequences of the integral representations 
proved in 
Proposition \ref{prop_IntsPolyPowsMultiples_OGF_pow_int-stmts_form_v2} 
applied to each of the series expansions for the 
higher-order derivative cases provided by 
Lemma \ref{lemma_JThetaFns_GenSeries_for_Higher-OrderDerivs-stmts_v1}. 
\end{proof} 

\section{Applications of Exponential Series Generating Functions} 
\label{Section_AppsOf_ExpSqSeries} 
\label{subSection_EGFs_and_GraphTheory_Examples} 
\label{subsubSection_Apps_in_Graph_Theory} 

\subsection{A Comparison of Characteristic Expansions of the Square Series Integrals} 
\label{subSection_EGFs_AComparison} 

%% See : "en-edges-labeled-graphs-OGF-EGF-summary-2013.04.05-v1.*": 
%% See : "temp-working-2014.01.31-v1.*": 
\begin{example}[The Number of Edges in Labeled Graphs] 
\label{prop_Number_of_Edges_LabeledGraph_OGF_EGF_forms} 
The \keywordemph{number of edges in a labeled graph on $n \geq 1$ nodes}, 
denoted by the sequence $\langle e(n) \rangle$, 
is given in closed-form by the formula \citep[A095351]{OEIS} 
\begin{align*} 
e(n) & = \frac{1}{4} n(n-1) 2^{n(n-1) / 2}.
\end{align*} 
The ordinary and exponential generating functions of this sequence, 
defined to be $e_{\sq}(z)$ and $\widehat{e}_{\sq}(z)$, respectively, 
correspond to the special cases of the 
second derivatives with respect to $z$ of the 
ordinary and exponential generating functions, 
$G_{\sq}(q, c, z)$ and $E_{\sq}(q, r, z)$, 
from Table \ref{table_Notation_SeriesDefsOfSpecialSqSeries-stmts_v1} 
which are expanded as 
\begin{align} 
\notag 
e_{\sq}(z) & := 
     \OGFSeriesTransform{z}{e(0),e(1),e(2),\ldots} \equiv 
     \frac{z^2}{4} \times G_{\sq}^{\prime\prime}\left( 
     2^{1/2}, 2^{-1/2}, z 
     \right) \\ 
\notag 
\widehat{e}_{\sq}(z) & := 
     \EGFSeriesTransform{z}{e(0),e(1),e(2),\ldots} \equiv 
     \frac{z^2}{4} \times E_{\sq}^{\prime\prime}\left( 
     2^{1/2}, 2^{-1/2}, z 
     \right). 
\end{align} 
The special cases of the OGF formulas cited in 
\eqref{eqn_Table_SpCaseSequenceOGFFormulas-GeomSeriesBased_v4} and 
\eqref{eqn_Table_SpCaseSequenceOGFFormulas-ExpSeriesBased_v2} of 
Table \ref{table_fnSeq_OGFs_and_jthDerivs_listings-stmts_v1} 
then lead to the next integral representations for each of these 
sequence generating functions given by 
\begin{align} 
\notag 
e_{\sq}(z) & = 
     \int_0^{\infty} \frac{e^{-t^2/2}}{\sqrt{2\pi}}\left[ 
     \frac{6 z^{4} - \sqrt{2} z^{3} (z^2+6) 
     \cosh\left(\sqrt{\Log(2)} t\right) + 4 z^{2} 
     \cosh\left(2 \sqrt{\Log(2)} t\right)}{ 
     \left(z^2 - 2\sqrt{2} z \cosh\left(\sqrt{\Log(2)} t\right) + 2\right)^3} 
     \right] dt \\ 
\notag 
\widehat{e}_{\sq}(z) & = 
     \frac{z^2}{8} \times \int_0^{\infty} 
     \frac{e^{-t^2/2}}{\sqrt{2\pi}}\left[ 
     \sum_{b=\pm 1} 
     \exp\left(e^{bt \sqrt{\Log(2)}} \frac{z}{\sqrt{2}} + 2bt \sqrt{\Log(2)} 
     \right) 
     \right] dt. 
\end{align} 
\end{example} 

The comparison given in 
Example \ref{prop_Number_of_Edges_LabeledGraph_OGF_EGF_forms} 
clearly identifies these characteristic, or at least stylistic, 
differences in the 
resulting square series integral representations derived from the forms of 
these separate ``\emph{ordinary}'' and ``\emph{exponential}'' sequence types. 
We also see the general similarities in form of the first 
geometric-series-like integrals to the Fourier series for the 
\emphonce{Poisson kernel}, where we point out the similarities of the second 
exponential-series-like integrals to generating functions for the 
non-exponential, Stirling-number-related \emphonce{Bell polynomials}, $B_n(x)$
\citep[\S 1.15(iii); \S 26.7]{NISTHB} \citep[\S 4.1.8]{UC}. 

\subsection{Initial Results} 

%% See : "c_nsquared_integral-EGF-2013.05.12-v2.*": 
%% See : "c_nsquared_integral-EGF-2013.05.12-v2.*": 
\begin{prop}[Exponential Square Series Generating Functions] 
\label{prop_EGF_for_OrdGeomSqSeries_IntRep} 
\label{cor_Exp-BasedSqSeries_Variant_qPowBinom_nChoose2} 
For any fixed parameters $q, z,r \in \mathbb{C}$, the 
\keywordemph{exponential square series} functions, defined respectively as 
\eqref{eqn_Esq_qrz_series_def_v1} and 
\eqref{eqn_ETildesq_qBinomPow_rz_int_rep_series_def_v1} in 
Table \ref{table_Notation_SeriesDefsOfSpecialSqSeries-stmts_v1}, 
have the following integral representations: 
\StartGroupingSubEquations 
\begin{align} 
\label{eqn_Esq_qrz_series_def_and_int_rep_v1} 
E_{\sq}(q, r, z) & = 
     \int_0^{\infty} \frac{e^{-t^2/2}}{\sqrt{2\pi}} \left[ 
     e^{e^{\sqrt{2 \Log(q)} t} rz} + e^{e^{-\sqrt{2 \Log(q)} t} rz} 
     \right] dt \\ 
\label{eqn_ETildesq_qBinomPow_rz_int_rep} 
\widetilde{E}_{\sq}(q, r, z) & = 
     \int_0^{\infty} \frac{e^{-t^2/2}}{\sqrt{2\pi}} \left[ 
     e^{e^{\sqrt{\Log(q)} t} \frac{rz}{\sqrt{q}}} + 
     e^{e^{-\sqrt{\Log(q)} t} \frac{rz}{\sqrt{q}}} 
     \right] dt. 
\end{align} 
\EndGroupingSubEquations 
\end{prop} 
\begin{proof} 
First, it is easy to see from 
Theorem \ref{thm_SqSeries_OGF_Transforms} 
applied to the first form of the exponential series OGF given in 
\eqref{eqn_Table_SpCaseSequenceOGFFormulas-ExpSeriesBased_v1}, that 
\begin{align*} 
E(q, r, z) & = 
     \int_0^{\infty} \frac{e^{-t^2/2}}{\sqrt{2\pi}} \left[ 
     \sum_{b = \pm 1} e^{e^{b \sqrt{2 \Log(q)} t} rz} 
     \right] dt, 
\end{align*} 
which implies the first result given in 
\eqref{eqn_Esq_qrz_series_def_and_int_rep_v1}. 
Next, the series for the second function, 
$\widetilde{E}_{\sq}(q, r, z)$, over the binomial powers of 
$q^{\binom{n}{2}} \equiv q^{n(n-1)/2}$ 
is expanded through first function as 
\begin{align*} 
\widetilde{E}_{\sq}(q, r, z) & = 
     \sum_{n=0}^{\infty} q^{n(n-1) / 2} r^n \frac{z^n}{n!} \equiv %\\ 
     E_{\sq}\left(q^{1/2}, r q^{-1/2}, z\right). 
\end{align*} 
This similarly leads to the form of the second result stated in 
\eqref{eqn_ETildesq_qBinomPow_rz_int_rep}. 
\end{proof} 

\subsection{Examples of Chromatic Generating Functions} 

A class of generating function expansions resulting from the application of 
Proposition \ref{prop_EGF_for_OrdGeomSqSeries_IntRep} 
to exponential-series-based OGFs defines 
\textit{chromatic generating functions} of the form 
\citep[\S 3.15]{ECV1} 
\begin{equation*} 
\widehat{F}_f(q, z) = \sum_{n=0}^{\infty} 
     \frac{f(n) z^n}{q^{\binom{n}{2}} n!}, 
\end{equation*} 
for some prescribed sequence of terms, $\langle f(n) \rangle$. 
If $f(n)$ denotes the 
\textit{number of labeled acyclic digraphs with $n$ vertices}, 
for example as considered in 
\citep[Prop. 2.1]{STANLEYACYCLICGRAPHS} \citep[\cf Ex. 3.15.1(e)]{ECV1}, 
then the chromatic generating function, $\widehat{F}_f(q,z)$, has the form 
\begin{equation*} 
\widehat{F}_{f}(2,z) := 
     \sum_{n=0}^{\infty} \frac{f(n) z^n}{2^{\binom{2}{n}} n!} = 
     \left(\sum_{n=0}^{\infty} \frac{(-1)^n z^n}{2^{\binom{n}{2}} n!} 
     \right)^{-1}. 
\end{equation*} 
Example \ref{ex_ChromaticGFs_ExpSqSeries_Apps} below 
cites other particular chromatic generating functions corresponding to 
special cases of the integral representations for the series 
involving binomial square powers of $q$ established in 
Proposition \ref{cor_Exp-BasedSqSeries_Variant_qPowBinom_nChoose2}. 

\begin{example}[Powers of a Special Chromatic Generating Function] 
\label{ex_ChromaticGFs_ExpSqSeries_Apps} 
Suppose that $G_n$ is a finite, simple graph with $n$ vertices and that 
$\chi(G_n, \lambda)$ denotes the \textit{chromatic polynomial} of 
$G_n$ evaluated at some $\lambda \in \mathbb{C}$ 
\citep[\S 1]{STANLEYACYCLICGRAPHS}. 
Then for a non-negative integer $k$, a variant of the 
EGF, $\widehat{M}_k(z)$, 
for the sequence of $M_n(k) = \sum_{G_n} \chi(G_n, k)$ 
is expanded in integer powers of a square series integrals as 
\citep[\S 2]{STANLEYACYCLICGRAPHS} 
\begin{equation} 
\label{eqn_ChromaticGF_Kth_Powers_ident_stmt_v1} 
\widehat{M}_k(z) = \sum_{n=0}^{\infty} \frac{M_n(k) z^n}{
     2^{\binom{n}{2}} n!} = \left(\sum_{n=0}^{\infty} 
     \frac{z^n}{2^{\binom{n}{2}} n!}\right)^{k}. 
\end{equation} 
The right-hand-side of \eqref{eqn_ChromaticGF_Kth_Powers_ident_stmt_v1} 
corresponds to the $k^{th}$ powers of the generating function, 
$E_{\sq}\left(1/2, 1, z\right)$, which is 
expanded for $k \in \mathbb{Z}^{+}$ through the integral given in 
Proposition \ref{cor_Exp-BasedSqSeries_Variant_qPowBinom_nChoose2} 
as 
%% See : "temp-working-2014.01.31-v1.*": 
\begin{equation} 
\label{eqn_ChromaticGF_Kth_Powers_ident_stmt_v2} 
\widehat{M}_k(z) = \left(
     \int_0^{\infty} \frac{e^{-t^2/2}}{\sqrt{2\pi}} \left[ 
     e^{e^{\imath \sqrt{\Log(2)} t} \sqrt{2} z} + 
     e^{e^{-\imath \sqrt{\Log(2)} t} \sqrt{2} z} 
     \right] dt 
     \right)^{k}. 
\end{equation} 
Notice that since the integrands in \eqref{eqn_ETildesq_qBinomPow_rz_int_rep} 
can be integrated termwise in $z$, the 
right-hand-side of \eqref{eqn_ChromaticGF_Kth_Powers_ident_stmt_v2} 
is also expressed as the multiple integral 
%% See : "temp-working-2014.01.31-v1.*": 
\begin{equation} 
\notag 
%\label{eqn_ChromaticGF_Kth_Powers_ident_stmt_v3} 
\widehat{M}_k(z) = \int_0^{\infty} \cdots \int_0^{\infty} 
     \frac{e^{-(t_1^2+ \cdots + t_k^2)/2}}{\left(2\pi\right)^{k/2}} 
     \left[\prod_{i=1}^{k} 
     e^{e^{\imath \sqrt{\Log(2)} t_i} \sqrt{2} z} + 
     e^{e^{-\imath \sqrt{\Log(2)} t_i} \sqrt{2} z}
     \right] dt_1 \cdots dt_k. 
\end{equation} 
\end{example} 

\subsection{Another Application: A Generalized Form of the Binomial Theorem} 

%% See : "sqseries-binomthm-analog-summary-2013.04.05-v1.*": 
%% See : "Mpn-GenMultinomial-like-finite-sums-summary-2013.05.20-v1.*": 
\begin{prop}[A Square Series Analog to the Binomial Theorem] 
\label{prop_BinomThm_SqSeries_analog} 
For constants $q, r, c, d \in \mathbb{C}$ and $n \in \mathbb{N}$, 
a generalized analog to the binomial theorem involving square powers of the 
parameters $q$ and $r$ has the following double integral representation: 
\begin{align} 
\label{eqn_SquareSeries_BinomThm_Analog_stmt} 
 & \sum_{k=0}^{n} \binom{n}{k} c^k q^{k^2} d^{n-k} r^{(n-k)^2} = 
     \int_0^{\infty} \int_0^{\infty} 
     \frac{e^{-(t^2+s^2) / 2}}{2\pi} \Biggl[ 
     \left(c e^{\sqrt{2 \Log(q)} t} + d e^{\sqrt{2 \Log(r)} s}\right)^n \\ 
\notag 
   & \phantom{\sum\ \ \binom{n}{k}} + 
     \left(c e^{\sqrt{2 \Log(q)} t} + d e^{-\sqrt{2 \Log(r)} s}\right)^n + 
     \left(c e^{-\sqrt{2 \Log(q)} t} + d e^{\sqrt{2 \Log(r)} s}\right)^n \\ 
\notag 
   & \phantom{\sum\ \ \binom{n}{k}} + 
     \left(c e^{-\sqrt{2 \Log(q)} t} + d e^{-\sqrt{2 \Log(r)} s}\right)^n
     \Biggr] dt ds. 
\end{align} 
\end{prop} 
\begin{proof} 
Let $c, q \in \mathbb{C}$ and for variable $t \in \mathbb{R}$ define the 
function, $E_q^{(t)}(c, z)$, by 
\begin{equation} 
\label{eqn_prop_BinomThmAnalog_Eqtcz_fn_def-stmt_v1} 
E_q^{(t)}(c, z) = e^{e^{\sqrt{2 \Log(q)} t} cz} + 
     e^{e^{-\sqrt{2 \Log(q)} t} cz}. 
\end{equation} 
It follows from the transformation result in 
Proposition \ref{prop_EGF_for_OrdGeomSqSeries_IntRep} that 
\begin{equation} 
\label{eqn_Eqtcz_series_integral_ident_def} 
\sum_{n=0}^{\infty} q^{n^2} c^n \frac{z^n}{n!} = \int_0^{\infty} 
     \frac{e^{-t^2/2}}{\sqrt{2\pi}} E_q^{(t)}(c, z) dt. 
\end{equation} 
Fix the constants $c, q, d, r \in \mathbb{C}$ and observe that the 
next double integral results for the coefficients in the 
discrete convolution of two power series in the form of 
\eqref{eqn_Eqtcz_series_integral_ident_def}. 
\begin{equation} 
\label{eqn_BinomThmAnalog_proof_scf_note_v1} 
\sum_{k=0}^{n} \frac{c^k q^{k^2}}{k!} \cdot 
     \frac{d^{n-k} r^{(n-k)^2}}{(n-k)!} = \int_0^{\infty} \int_0^{\infty} 
     \frac{e^{-(t^2+s^2)/2}}{2\pi} \cdot [z^n]\left( 
     E_q^{(t)}(c, z) E_r^{(s)}(d, z)\right) dt ds 
\end{equation} 
The left-hand-side sum in \eqref{eqn_SquareSeries_BinomThm_Analog_stmt} 
is obtained by multiplying the coefficient definition 
on the right-hand-side the previous equation by a factor of $n!$. 
This operation can be applied termwise to the series for 
$E_q^{(t)}(c, z) E_r^{(s)}(d, z)$ through the next integral for the 
single factorial function, or gamma function, where 
$\Gamma(n+1) = n!$ whenever $n \in \mathbb{N}$ 
\citep[\S 5.2(i); \S 5.4(i)]{NISTHB}. 
\begin{equation} 
\label{eqn_GammaFn_nFact_integral_stmt} 
\Gamma(n+1) = \int_0^{\infty} u^n e^{-u} du,\ 
     \text{ $\RePart(n) > -1$ } 
\end{equation} 
Then it is easy to see that 
\begin{equation} 
\label{eqn_BinomThm_analog_proof_Expeabuz_integral_result} 
\int_0^{\infty} e^{e^{a} buz} e^{-u} du = \int_0^{\infty} 
     e^{-(1-e^{a} b z) u} du = \frac{1}{\left(1 - e^{a} b z\right)} 
\end{equation} 
for constants $a, b$ and $z$ such that the right--hand--side of 
\eqref{eqn_BinomThm_analog_proof_Expeabuz_integral_result} 
satisfies $\RePart(e^{a} b z) < 1$. 
\newcommand{\bcfij}[8]{\ensuremath{b_{#1,#2}^{\left(#3,#4\right)}\left(#5, #6, #7, #8\right)}} 
Next, 
let the coefficient terms, $\bcfij{i}{j}{t}{s}{c}{q}{d}{r}$, be defined in the 
form of the next equation. 
\begin{equation*}
\bcfij{i}{j}{t}{s}{c}{q}{d}{r} = 
     c \cdot \exp\left((-1)^{i} \sqrt{2 \Log(q)} t\right) + 
     d \cdot \exp\left((-1)^{j} \sqrt{2 \Log(r)} s\right) 
\end{equation*} 
By combining the results in 
\eqref{eqn_GammaFn_nFact_integral_stmt} and 
\eqref{eqn_BinomThm_analog_proof_Expeabuz_integral_result}, it follows that 
the function 
\begin{align*} 
\widetilde{E}_{q,r}^{(t,s)}(c, d, z) & := 
     \int_0^{\infty} \frac{e^{-(t^2+s^2)/2}}{2\pi} 
     E_q^{(t)}(c, uz) E_r^{(s)}(d, uz) e^{-u} du \\ 
   & = 
     \frac{e^{-(t^2+s^2)/2}}{2\pi}\left[ 
     \sum_{(i,j) \in \{0,1\} \times \{0,1\}} 
     \frac{1}{(1 - \bcfij{i}{j}{t}{s}{c}{q}{d}{r} z)}
     \right] \\ 
   & = 
     \sum_{n=0}^{\infty} \left[ 
     \sum_{(i,j) \in \{0,1\} \times \{0,1\}} 
     \frac{e^{-(t^2+s^2)/2}}{2\pi} \bcfij{i}{j}{t}{s}{c}{q}{d}{r}^{n} 
     \right] z^n. 
\end{align*} 
The complete result given in \eqref{eqn_SquareSeries_BinomThm_Analog_stmt} 
then follows from the last equation by integrating over non-negative 
$s, t \in \mathbb{R}$ as 
\begin{equation} 
\notag 
%\label{eqn_BinomThmAnalog_proof_scf_note_v2} 
\sum_{k=0}^{n} \binom{n}{k} c^k q^{k^2} d^{n-k} r^{(n-k)^2} = 
     \int_0^{\infty} \int_0^{\infty} [z^n]\left( 
     \widetilde{E}_{q,r}^{(t,s)}(c, d, z)\right) dt ds. 
     \qedhere  
\end{equation} 
\end{proof} 

\section{Direct Expansions of Fourier-Type Square Series} 
\label{Section_ExpansionsOf_Fourier-TypeSqSeries} 

\subsection{Initial Results} 

\StartGroupingSubEquations{} 

\begin{cor}[Integral Representations of Fourier--Type Series] 
\label{cor_FscAlphaBetauqz_FourierSquareSeries-scFnsCosSinIntReps} 
For $\alpha, \beta \in \mathbb{R}$, and $c, z \in \mathbb{C}$ with 
$|cz| < 1$, the generalized 
\emph{Fourier-type square series functions} defined by the series in 
\eqref{eqn_FscAlphaBetauqz_FourierSquareSeries-introc_def_v3} 
have the following integral representations: 
\begin{align} 
\label{eqn_FscAlphaBetauqz_FourierSquareSeries-scFnCosIntRep-stmt_v1} 
 & F_{\cos}\left(\alpha, \beta; q, c, z\right) = 
     \int_0^{\infty} \frac{e^{\imath\beta} e^{-t^2/2}}{\sqrt{2\pi}} \left[ 
     \frac{1 - e^{\imath\alpha} cz \cosh\left(t \sqrt{2 \Log(q)}\right)}{ 
     e^{2\imath\alpha} c^2 z^2 - 2 e^{\imath\alpha} cz 
     \cosh\left(t \sqrt{2 \Log(q)}\right) + 1} 
     \right] dt \\ 
\notag 
   & \phantom{F_{\cos}\bigl(\alpha, \beta; q, c } + 
     \int_0^{\infty} \frac{e^{-\imath\beta} e^{-t^2/2}}{\sqrt{2\pi}} \left[ 
     \frac{1 - e^{-\imath\alpha} cz \cosh\left(t \sqrt{2 \Log(q)}\right)}{ 
     e^{-2\imath\alpha} c^2 z^2 - 2 e^{-\imath\alpha} cz 
     \cosh\left(t \sqrt{2 \Log(q)}\right) + 1} 
     \right] dt \\ 
\label{eqn_FscAlphaBetauqz_FourierSquareSeries-scFnSinIntRep-stmt_v2} 
 & F_{\sin}\left(\alpha, \beta; q, c, z\right) = 
     \int_0^{\infty} \frac{e^{\imath\beta} e^{-t^2/2}}{\sqrt{2\pi}\imath} 
     \left[ 
     \frac{1 - e^{\imath\alpha} cz \cosh\left(t \sqrt{2 \Log(q)}\right)}{ 
     e^{2\imath\alpha} c^2 z^2 - 2 e^{\imath\alpha} cz 
     \cosh\left(t \sqrt{2 \Log(q)}\right) + 1} 
     \right] dt \\ 
\notag 
   & \phantom{F_{\sin}\bigl(\alpha, \beta; q, c } - 
     \int_0^{\infty} \frac{e^{-\imath\beta} e^{-t^2/2}}{\sqrt{2\pi}\imath} 
     \left[ 
     \frac{1 - e^{-\imath\alpha} cz \cosh\left(t \sqrt{2 \Log(q)}\right)}{ 
     e^{-2\imath\alpha} c^2 z^2 - 2 e^{-\imath\alpha} cz 
     \cosh\left(t \sqrt{2 \Log(q)}\right) + 1} 
     \right] dt. 
\end{align} 
\end{cor} 
\begin{proof} 
We prove the forms of these two integral representations by first 
expanding the trigonometric function sequences as follows 
\citep[\S 4.14]{NISTHB}: 
\begin{align} 
\notag 
\cos\left(\alpha n+\beta\right) \cdot c^{n} & = 
     \frac{e^{\imath\beta}}{2} \cdot \left(e^{\imath\alpha} c\right)^{n} + 
     \frac{e^{-\imath\beta}}{2} \cdot \left(e^{-\imath\alpha} c\right)^{n} \\ 
\notag 
\sin\left(\alpha n+\beta\right) \cdot c^{n} & = 
     \frac{e^{\imath\beta}}{2\imath} \cdot 
     \left(e^{\imath\alpha} c\right)^{n} - 
     \frac{e^{-\imath\beta}}{2\imath} \cdot 
     \left(e^{-\imath\alpha} c\right)^{n}. 
\end{align} 
Then since $|e^{\pm \imath\alpha}| \equiv 1$ whenever $\alpha \in \mathbb{R}$, 
whenever $|e^{\pm \imath\alpha} cz| \equiv |cz| < 1$, the 
generating functions for each of these sequences are expanded in terms of the 
geometric series sequence OGFs in 
\eqref{eqn_Table_SpCaseSequenceOGFFormulas-GeomSeriesBased_v1} as 
\begin{align} 
\notag 
F_{\cos}\left(\alpha, \beta; 1, c, z\right) & = 
     \frac{e^{\imath\beta}}{2} \cdot 
     G_{\sq}\left(1, e^{\imath\alpha} c, z\right) + 
     \frac{e^{-\imath\beta}}{2} \cdot 
     G_{\sq}\left(1, e^{-\imath\alpha} c, z\right) \\ 
\notag 
F_{\sin}\left(\alpha, \beta; 1, c, z\right) & = 
     \frac{e^{\imath\beta}}{2\imath} \cdot 
     G_{\sq}\left(1, e^{\imath\alpha} c, z\right) - 
     \frac{e^{-\imath\beta}}{2\imath} \cdot 
     G_{\sq}\left(1, e^{-\imath\alpha} c, z\right). 
\end{align} 
The pair of integral formulas stated as the results in 
\eqref{eqn_FscAlphaBetauqz_FourierSquareSeries-scFnCosIntRep-stmt_v1} and 
\eqref{eqn_FscAlphaBetauqz_FourierSquareSeries-scFnSinIntRep-stmt_v2} 
are then obtained as the particular special cases of 
Proposition \ref{prop_OrdExp_GeomSquareSeries} 
corresponding to the expansions in the previous equations. 
\end{proof} 

%% See : "scterms-direct-sqseries-summary-2013.05.12-v1.*": 
%% See : "temp-working-2014.07.25-v1.*": 
\begin{remark} 
\label{remark_FourierTypeSqSeries_AltCompactExps} 
Notice that when $\alpha, \beta \in \mathbb{R}$, the 
forms of the Fourier series OGFs, 
$F_{\scFn}\left(\alpha, \beta; 1, c, z\right)$ for each of the functions 
$\scFn \in \{\cos, \sin\}$, may be expressed in a 
slightly more abbreviated form by 
\begin{align} 
\notag 
F_{\scFn}\left(\alpha, \beta; 1, c, z\right) & = 
     \frac{e^{\imath\alpha} \scFn(\beta-\alpha) - \scFn(\beta)}{ 
     \left(e^{2\imath\alpha} - 1\right) \cdot 
     \left(1 - e^{-\imath\alpha} cz\right)} + 
     \frac{e^{\imath\alpha} \left(e^{\imath\alpha} \scFn(\beta) - 
     \scFn(\beta-\alpha)\right)}{ 
     \left(e^{2\imath\alpha} - 1\right) \cdot 
     \left(1 - e^{\imath\alpha} cz\right)} 
\end{align} 
The result in the corollary may then be expanded through this alternate form 
of the ordinary generating function as 
\begin{align} 
\label{eqn_FscAlphaBetauqz_FourierSquareSeries-scFnCosSinGenIntRep-stmt_v3} 
F_{\scFn}\left(\alpha, \beta; q, c, z\right) & = 
     \int_0^{\infty} \frac{e^{-t^2/2}}{\sqrt{2\pi}} \Biggl[ 
     \sum_{b=\pm 1} 
     \frac{2b \cdot e^{\imath\alpha} \left( 
     e^{b\imath\alpha} \scFn(\beta) - \scFn(\beta-\alpha) 
     \right)}{\left(e^{2\imath\alpha} - 1\right)} \times \\ 
\notag 
   & \phantom{= \int_0^{\infty} \frac{e^{-t^2/2}}{\sqrt{2\pi}} \Biggl[} \times 
     \frac{1 - e^{b \imath\alpha} cz \cosh\left(t \sqrt{2 \Log(q)}\right)}{ 
     e^{2b\imath\alpha} c^2 z^2 - 2 e^{b\imath\alpha} cz 
     \cosh\left(t \sqrt{2 \Log(q)}\right) + 1} 
     \Biggr] dt, 
\end{align} 
for any $\alpha, \beta \in \mathbb{R}$ and whenever $c,z \in \mathbb{C}$ 
are chosen such that $|cz| < 1$. 
\end{remark} 

\EndGroupingSubEquations{} 

\subsection{Fourier Series Expansions of the Jacobi Theta Functions} 

Exact new integral representations for the asymmetric, unilateral 
Fourier series expansions of the Jacobi theta functions defined in the 
introduction are expressed through the results in 
Corollary \ref{cor_FscAlphaBetauqz_FourierSquareSeries-scFnsCosSinIntReps} and 
Remark \ref{remark_FourierTypeSqSeries_AltCompactExps} as follows 
\citep[\S 20.2(i)]{NISTHB}: 
\begin{align*} 
\vartheta_1(u, q) & = 2 q^{1/4} F_{\sin}\left(2u, u; q, q, -1\right) \\ 
\vartheta_2(u, q) & = 2 q^{1/4} F_{\cos}\left(2u, u; q, q, 1\right) \\ 
\vartheta_3(u, q) & = 1 + 2 F_{\cos}\left(2u, 0; q, 1, 1\right) \\ 
\vartheta_4(u, q) & = 1 - 2q F_{\cos}\left(2u, 2u; q, q^2, -1\right) 
\end{align*} 
We do not provide the explicit integral representations for each of these 
theta function variants in this section since these expansions 
follow immediately by substitution of the integral formulas from the 
last subsection above. 

\section{Conclusions} 

\subsection{Summary} 

We have proved the forms of several new integral representations for 
geometric-series-type square series transformations, 
exponential-series-type square series transformations, and 
Fourier-type square series transformations. 
Specific applications of the new results in the article include special case 
integrals for special infinite products, theta functions, and many 
examples of new single and double integral representations for other special 
constant values and identities. 
The particular forms of the explicit special case 
applications cited within the article are easily extended to enumerate 
many other series identities and special functions that arise in practice. 

One key aspect we have not discussed within the article is the relations 
of the geometric and Fourier-type transformation integrals to other 
Fourier series expansions. 
We remarked in 
Section \ref{subSection_EGFs_and_GraphTheory_Examples} 
about the similarity of the geometric square series characteristic 
expansions to the Poisson kernel \citep[\S 1.15(iii)]{NISTHB}. 
Two other Fourier series expansions related to these integral forms 
are given by 
\begin{align*} 
\sum_{k \geq 0} \frac{\sin(kx)}{r^k} & = 
     \frac{r \sin(x)}{r^2 - 2r \cos(x) + 1} \\ 
\sum_{k \geq 1} \frac{\cos(kx)}{r^k} & = 
     \frac{r \cos(x) - 1}{r^2 - 2r \cos(x) + 1}. 
\end{align*} 
One possible topic of future work on these transformations is to 
consider the relations of the square series generating function transformations 
to known Fourier series expansions, identities, and transforms -- 
even when the integral depends on a formal power series parameter $z$ 
which is independent of the squared series parameter $q$. 

\subsection{Comparisons to the Weierstrass Elliptic Functions} 

Perhaps the most striking similarities to the integral representations 
offered by the results 
in \sref{Section_Applications_of_GeomSqSeries} and 
Section \ref{Section_ExpansionsOf_Fourier-TypeSqSeries} 
of this article are found in the \emph{NIST Handbook of Mathematical Functions} 
citing a few relevant properties of the 
\keywordemph{Weierstrass elliptic functions}, $\wp(z)$ and $\zeta(z)$ 
\citep[\S 23; \S 23.11]{NISTHB}. 
The similar integral forms of interest satisfied by these 
elliptic functions are expanded through the auxiliary functions, 
$f_1(s, \tau)$ and $f_2(s, \tau)$, defined by 
\eqref{eqn_WEllipticFns_AuxIntFns_f1f2_stau_defs} below. 
The particular integral representations for the functions, 
$\wp(z)$ and $\zeta(z)$, are restated in 
\eqref{eqn_WEllipticFns_IntReps_ReStmts-v1} and 
are then compared to the 
forms of the integrals established in 
Proposition \ref{prop_OrdExp_GeomSquareSeries} and 
Remark \ref{prop_GenGeomSqSeries_integral_stmt} 
of this article. 

Let the parameter $\tau := \omega_3 / \omega_1$ for fixed 
$\omega_1, \omega_3 \in \mathbb{C} \setminus \{0\}$, and define the 
functions, $f_1(s, \tau)$ and $f_2(s, \tau)$, 
over non--negative $s \in \mathbb{R}$ as in the following equations 
\citep[eq. (23.11.1); \S 23.11]{NISTHB}: 
\begin{align} 
\label{eqn_WEllipticFns_AuxIntFns_f1f2_stau_defs} 
f_1(s, \tau) & := \frac{\cosh^2\left(\frac{1}{2} \tau s\right)}{ 
     e^{-2s} - 2 e^{-s} \cosh\left(\tau s\right) + 1} \\ 
\notag 
f_2(s, \tau) & := \frac{\cos^2\left(\frac{1}{2} s\right)}{ 
     e^{2 \imath\tau s} - 2 e^{\imath\tau s} \cos\left(s\right) + 1}. 
\end{align} 
Provided that both $-1 < \Re(z + \tau) < 1$ and $|\Im(z)| < \Im(\tau)$, 
these elliptic functions have the integral representations 
given in terms of the auxiliary functions from 
\eqref{eqn_WEllipticFns_AuxIntFns_f1f2_stau_defs} 
stated in the respective forms of the next equations 
\citep[\S 23.11]{NISTHB}. 
\begin{align} 
\label{eqn_WEllipticFns_IntReps_ReStmts-v1} 
\wp(z) & = \frac{1}{z^2} + 8 \times \int_0^{\infty} s \left[ 
     e^{-s} \sinh^2\left(\frac{zs}{2}\right) f_1(s, \tau) + 
     e^{\imath\tau s} \sin^2\left(\frac{zs}{2}\right) f_2(s, \tau) 
     \right] ds \\ 
\notag 
\zeta(z) & = \frac{1}{z} + \int_0^{\infty} \left[ 
     e^{-s} \left(zs - \sinh(zs)\right) f_1(s, \tau) - 
     e^{\imath\tau s} \left(zs - \sinh(zs)\right) f_2(s, \tau) 
     \right] ds 
\end{align} 
Notice that the first function, $f_1(s, \tau)$, defined in 
\eqref{eqn_WEllipticFns_AuxIntFns_f1f2_stau_defs} is similar in form to 
many of the identities and special function examples cited as 
applications in \sref{Section_Applications_of_GeomSqSeries} and 
\sref{Section_ExpansionsOf_Fourier-TypeSqSeries}. 
In this case, the parameter $q$ corresponds to an exponential function of 
$\tau$. 
The second definition of the function, $f_2(s, \tau)$, given in terms of the 
cosine function is also similar in form to the integrands that 
result from the computations of the explicit special values of 
Ramanujan's functions, $\varphi(q)$ and $\psi(q)$, 
derived as applications in 
Corollary \ref{cor_VarPhi_JTheta3_Integrals_for_SpValues} and 
Corollary \ref{cor_RamPsi_JTheta2_Integrals_for_SpValues} 
in \sref{subsubSection_Apps_SpCases_of_the_JacobiThetaFns} 
\citep[\cf \S 24.7(ii)]{NISTHB}. 

\subsection{Some Limitations of the Geometric-Series-Based Transformations} 

Suppose that $q, z(q) \in \mathbb{C}$ are selected such that 
$|q| < 1$ and $z(q) \neq 0$. Then the bilateral square series, 
$B_{\sq}(q, z)$ and $B_{a,b}(r_, s; q, z)$ defined below, 
converge and have resulting unilateral series expansions 
\citep[\cf \S 19.8]{HARDYWRIGHTNUMT} of the form 
\begin{align} 
\label{eqn_Bsqcqz_Bilateral_SqSeries_Variant-intro_stmt_v1} 
B_{\sq}(q, z) & := \sum_{n=-\infty}^{\infty} q^{n^2} z^n \\ 
\notag 
   & = 
     1 + \sum_{n=1}^{\infty} q^{n^2} \left(z^{n} + z^{-n}\right) \\ 
\label{eqn_Babrsqz_Bilateral_SqSeries_Variant-intro_stmt_v2} 
B_{a,b}\left(r, s; q, z\right) & := \sum_{n=-\infty}^{\infty} 
     (-1)^{n} (an+b) q^{n(rn+s)} z^n \\ 
\notag 
   & \phantom{:} = 
     \sum_{n=0}^{\infty} (-1)^n (an+b) q^{n(rn+s)} z^n - 
     \sum_{n=1}^{\infty} (-1)^{n} (an-b) q^{n(rn-s)} z^{-n}. 
\end{align} 
Notice that even though both of the series defined in 
\eqref{eqn_Bsqcqz_Bilateral_SqSeries_Variant-intro_stmt_v1} and 
\eqref{eqn_Babrsqz_Bilateral_SqSeries_Variant-intro_stmt_v2} 
converge for all $z(q) \in \mathbb{C}$ whenever $|q| < 1$, the 
square series integrals derived in 
Section \ref{Section_Applications_of_GeomSqSeries} 
cannot be applied 
directly to expand the forms of these bilateral series. 
This results from the constructions of the geometric-series-based 
transformations which are derived from the OGF, $F(z) \equiv (1-z)^{-1}$, 
which only converges when $|z| < 1$, and not in the 
symmetric series case where $|1/z| > 1$. 

\renewcommand{\refname}{References} 
\bibliographystyle{amsalpha} 

\providecommand{\bysame}{\leavevmode\hbox to3em{\hrulefill}\thinspace}
\providecommand{\MR}{\relax\ifhmode\unskip\space\fi MR }
\providecommand{\MRhref}[2]{%
  \href{http://www.ams.org/mathscinet-getitem?mr=#1}{#2}
}

\end{document}